\documentclass[11pt]{amsart}
\usepackage{amssymb,amsmath,amsthm}
\usepackage{a4wide}
\usepackage{graphicx}

\usepackage{color}
\usepackage{mathrsfs}
\usepackage{latexsym}
\usepackage{bbm}

\numberwithin{equation}{section}

\theoremstyle{plain}
\begingroup
\newtheorem{theorem}{Theorem}[section]
\newtheorem{lemma}[theorem]{Lemma}
\newtheorem{proposition}[theorem]{Proposition}

\endgroup

\theoremstyle{definition}
\begingroup
\newtheorem{definition}[theorem]{Definition}
\newtheorem{remark}[theorem]{Remark}

\endgroup

\newcommand{\Reg}{\mathfrak{C}}
\newcommand{\regu}[1][]{{\ell #1,\beta}}

\newcommand{\F}{\mathcal{F}}

\newcommand{\N}{\mathbb{N}}

\newcommand{\R}{\mathbb{R}}

\newcommand{\Rd}{\R^d}
\newcommand{\ffi}{\varphi}
\newcommand{\e}{\varepsilon}

\newcommand{\ud}{\,\mathrm{d}}
\def\H{\mathscr{H}}
\def\P{\mathcal{P}}
\def\hh{\mathrm{H}}
\def\D{\mathrm{D}}
\def\kk{\mathcal{K}}
\def\J{\mathcal{J}}

\def\oc{\overline{c}}
\def\uc{\underline{c}}
\newcommand{\RT}{\R^d\times [0,+\infty)}
\newcommand{\dia}{\operatorname{\text{diam}}}
\newcommand{\osp}{\omega^{\mathrm {sp}}}
\newcommand{\oti}{\omega^{\mathrm {ti}}}
\newcommand{\ep}{\varepsilon}
\newcommand{\inn}{\mathrm{in}}
\newcommand{\ou}{\mathrm{ou}}
\newcommand{\rr}{\mathcal{R}}

\title[Stability of nonlocal curvature flows] {Stability results for nonlocal geometric evolutions and limit cases for fractional mean curvature flows}

\author[A. Cesaroni]
{A. Cesaroni}
\address[Annalisa Cesaroni]
{Dipartimento di Scienze Statistiche, Universit\`a di Padova, Via Cesare Battisti 241/243, I-35121 Padova, Italy}
\email{annalisa.cesaroni@unipd.it}

\author[L. De Luca]
{L. De Luca}
\address[Lucia De Luca]{IAC-CNR, Via dei Taurini 19, I-00184 Roma, Italy}
\email[L. De Luca]{lucia.deluca@cnr.it}

\author[M. Novaga]
{M. Novaga}
\address[Matteo Novaga]{Dipartimento di Matematica, Universit\`a di Pisa, Largo Bruno Pontecorvo 5, I-56127 Pisa, Italy}
\email[M. Novaga]{matteo.novaga@unipi.it}

\author[M. Ponsiglione]
{M. Ponsiglione}
\address[Marcello Ponsiglione]{Dipartimento di Matematica ``Guido Castelnuovo'', Sapienza Universit\`a di Roma, Piazzale Aldo Moro 2, I-00185 Roma, Italy
}
\email[M. Ponsiglione]{ponsigli@mat.uniroma1.it}
\begin{document}

\begin{abstract}
We introduce a notion of uniform convergence for local and nonlocal curvatures. 
Then, we propose an abstract method to prove the 
convergence of the corresponding geometric flows, within the  level set formulation. We  apply such a general theory to characterize the limits of $s$-fractional mean curvature flows as $s\to 0^+$ and $s\to 1^-$.  In analogy with the $s$-fractional mean curvature flows,  we introduce the notion of $s$-Riesz curvature flows and characterize its limit as $s\to 0^-$.  Eventually, we discuss the limit behavior as $r\to 0^+$ of the flow generated by a regularization of the $r$-Minkowski content. 
\end{abstract}

\maketitle

 \tableofcontents

\section{Introduction} In recent years, several examples of geometric flows
arising in different models 
from Physics, Biology and Materials science, have been considered and studied.
Without trying  to be exhaustive,
some relevant examples of these evolutions are the classical mean curvature flow, together with its anisotropic versions
 (see \cite{man} and the references therein),
the fractional mean curvature flow  \cite{i, jlm, cdnv, cnr},  nonlocal evolutions driven by suitable  
interaction kernels  emerging  as models for dislocation dynamics in 
crystals  \cite{c, dlfm}.  

It is well known that geometric evolutions may exhibit  singularities even starting from smooth initial data, so that, together with classical solutions, much effort has been devoted to 
develop weak formulations taking into account topology changes and providing global solutions. 

Here we focus on the so-called {\it level set} formulation  \cite{OS, ES, CGG}. Specifically,  
 we follow the approach in \cite{cmp}, which provides a unified  framework to deal with a  wide class of local and nonlocal 
geometric  flows. 
There, the authors  introduced a class of generalized  curvatures $\H$ defined on pairs $(x,E)$ where $E$ is a smooth set and $x\in\partial E$,    satisfying suitable structural assumptions, accounting continuity properties, monotonicity and translational invariance.  For such a class of curvatures, the authors   provided global existence and uniqueness of
generalized geometric evolutions of sets with compact boundary, i.e., of continuous viscosity  solutions of the 
corresponding level set equation  
\begin{equation}
\label{lintro}
\begin{cases}
 \partial_t u(x,t) + |\D u(x,t)| \H(x,  \{ y: \, u(y,t)\ge u(x,t)\}) \,=\,0,\\ 
 u(\cdot,0) = u_0(\cdot)\,,
\end{cases}
\end{equation} 
where   $u_0:\R^d\to\R$ is a continuous function, which is constant outside a compact set.  
We shall revisit this theory in Section \ref{seclevel}. 

In this paper we propose an abstract method to prove stability of geometric flows with respect to curvature variations,    within the same general framework introduced  in \cite{cmp}.  
We present the method  in Section \ref{general}.  First of all, in Definition \ref{curvconver}, we introduce a notion of uniform convergence of a family $\{\H^n\}_{n\in\N}$ of  curvatures to a limit curvature  $\H^\infty$, where $\H^n, \H^\infty$ are generalized curvatures in the sense of \cite{cmp}.   
By enforcing  uniform bounds on the  velocity of evolving balls, we recover  uniform continuity estimates of the corresponding level set solutions, which provide  the desired compactness properties for the viscosity solutions $u^n$ to \eqref{lintro} with $\H$ replaced by $\H^n$. Moreover, using the uniform convergence of the curvatures, and the appropriate notion of test functions for \eqref{lintro}, suitable to deal with the singularities of the operators at points where  $\D u=0$, we
recover the uniform convergence of the differential operators appearing in  \eqref{lintro}; we conclude that the (whole) sequence $\{u^n\}_{n\in\N}$  locally uniformly converges to the unique solution $u^\infty$ to \eqref{lintro} with $\H=\H^\infty$. 
This is the main abstract tool of the paper, precisely stated  in Theorem \ref{genthm}.
 
We apply such a result to characterize the limit cases of several parametrized families of geometric flows.  
First, 
in Section  \ref{fraccurv} we analyze the asymptotics of the (reparametrized in time) $s$-fractional mean curvature flow,
as  $s\to 1$ and $s\to 0$. The limit case $s\to 1$ is actually well understood: as pointed out in \cite{i}, suitably reparametrized in time solutions of the $s$-fractional mean curvature flow converge to the classical mean curvature flow. The case $s\to 0$ is, to our knowledge, completely new and, in some respects, more intriguing.  
Indeed,  
 the $s$-fractional perimeters, multiplied by $s$, converge (up to a prefactor) to the Lebesgue measure  \cite{ms, dfpv}, suggesting, at a first glance, that the corresponding reparametrized  flows converge to a trivial geometric evolution where  sets move  with constant normal velocity. Such a trivial motion reflects the degeneracy of the limit behaviour of the rescaled fractional perimeters as $s\to 0$. Our methods  provide a rigorous proof of these facts. 

Furthermore, very recently the next order expansion of $s$-fractional perimeters as $s\to 0$  has been obtained in \cite{dnp} in terms of $\Gamma$-convergence  (see \cite{cn} for the limit $s\to 1$). The limit perimeter, referred to as {\it $0$-fractional perimeter}, is now much less degenerate, enjoying for instance the fractional isoperimetric inequality, which establishes that balls are the only minimizers of  the $0$-fractional perimeter under a volume constraint.
 In this paper we compute the first variation of such a $0$-fractional perimeter, referred to as {\it $0$-fractional curvature} (see Definition \ref{zerocurv}),
 we prove that the $0$-fractional curvature is a generalized curvature in the sense of \cite{cmp}, so that the general results about existence and uniqueness of the level set flow obtained in \cite{cmp} apply. Remarkably, we show that this is the relevant object to describe the limit of $s$-fractional mean curvature flows as $s\to 0$. 
In fact, we first  show that the next  order asymptotics of the $s$-fractional  curvature as $s\to 0$ is exactly the $0$-fractional curvature; then, as a byproduct of our general stability result, we prove that solutions of suitably forced $s$-fractional mean curvature flows  converge to the solution of the $0$-fractional mean curvature flow. 
Here the appropriate forcing term is nothing but $1/s$ times the constant normal velocity appearing in the leading term of the Taylor expansion of the $s$-fractional curvature.

Our method is, as a matter of fact, very robust, and applies to several different contexts. To illustrate that, 
in analogy with the $s$-fractional  curvature,  we  introduce in Section  \ref{secriesz} the notion of  {\it $s$-Riesz curvature} $\mathcal{K}^s$, for $s<0$, as the first variation of  Riesz-type energies, see \eqref{rieszvar}.
We show that $\mathcal{K}^s$  is a generalized curvature in the  sense of \cite{cmp}, and hence the  geometric flow driven by $\mathcal{K}^s$  is globally well-defined; moreover, we study the first and second order limits of the rescaled flow as $s\to 0$, obtaining the same limit flows 
of the rescaled  $s$-fractional curvature flow.

Finally, in  Section \ref{secmin}  we consider the flow  generated by a suitable regularization of the $r$-Minkowski content, which is the measure of the $r$-neighborhood of the boundary of the set, divided by $2r$, see\eqref{defmink}. 
The $r$-Minkowski content has been introduced in \cite{b} in the context of  image denoising and then applied to vessel segmentation \cite{z}.
In order to get well-posedness of the corresponding geometric flow, in \cite{cmp,CMP0} the authors introduced a regularized version of the $r$-Minkowski content, see \eqref{varmr} for a precise definition, whose first variation is a generalized curvature  in the sense of \cite{cmp}. We  prove that the limit of such a geometric flow as $r\to 0$ is the classical mean curvature flow.
\vskip2pt
As already mentioned, the problem of convergence of rescaled nonlocal curvature flows  to local 
(isotropic and anisotropic) mean curvature flows has already been faced in the literature; see for instance \cite{i, dlfm, cp}.  One of the merits of our abstract approach is the capability to detect the asymptotic behavior of geometric evolutions even when the limit itself has a nonlocal character; this is actually the case of the $0$-fractional mean curvature flow.

It would be interesting to 
apply our general method to
 recover also the results in \cite{dlfm,cp}, 
and to extend our formalism to other weak notion  of solutions; 
 one could consider for instance the method of geometric barriers by De Giorgi \cite{dg} (adopted in \cite{cp}),  
which turns out to be equivalent in many cases to the level set formulation, at least for local evolutions (see \cite{bn1}).

Finally, our method could be exploited to study limit cases of many other  geometric evolutions. As a relevant example, we mention the limit of geometric flows driven by the $p$-capacity as $p\to 1$. Incidentally,  the viscosity approach proposed in this paper could turn out to be convenient also  to deal  with non-geometric equations such as nonlinear versions of $s$-fractional heat equations.

\vskip5pt
\noindent

{\bf Acknowledgments:}  The authors are members of the Gruppo Nazionale per l'Analisi Matematica, la Probabilit\`a e le loro Applicazioni (GNAMPA) of the Istituto Nazionale di Alta Matematica (INdAM).

\section{Level set formulation}\label{seclevel} 
In this section we revisit the level set formulation of generalized nonlocal curvature flows introduced in \cite{cmp}.
We start by providing the appropriate notion of nonlocal curvature.


\subsection{Axioms of nonlocal curvature}\label{axioms}

Let $\Reg$ be the class of  subsets of $\R^d$, which can be obtained as the closure of an open set with  compact $C^{\regu}$ boundary.
Throughout the paper
$\ell\ge 2$ and $\beta\in [0,1]$ will be fixed; the reader may simply
assume $\ell=2$, $\beta=0$.

We will deal with ``curvature''  functions $x\mapsto \H(x,E)\in\R$ defined for $E\in\Reg$ and $x\in\partial E$. 

\begin{definition}\label{defcu}
We say that $\H$ is a nonlocal curvature if it satisfies the standing assumptions (M), (T)  and (C) below.
\begin{itemize}
\item[(M)] Monotonicity: If 
$E, F\in\Reg$ with $E\subseteq F$, and if $x\in \partial F\cap \partial E$,
then $\H(x,F)\le\H(x,E)$;
\item[(T)] Translational invariance: for any $E\in\Reg$, $x\in\partial E$,
$y\in\R^d$, $\H(x,E)=\H(x+y,E+y)$;
\item[(C)] Continuity:
If $\{E_n\}_{n\in\N}\subset\Reg$, $E\in\Reg$, and 
$E_n\to E$ in $\Reg$, then $\H(x,E_n)\to \H(x,E)$ for every $x\in\partial E_n\cap\partial E$.
\end{itemize}
\end{definition}

{
Here and throughout the paper, by  $E_n\to E$ in $\Reg$ we mean
that there exists a sequence of diffeomorphisms
$\{\Phi_n \}_{n\in\N}$ converging to the identity in $C^{\ell,\beta}$, with  $E_n=\Phi_n (E)$.
}

By assumption (C), for any $\rho>0$ we can define the quantities
\begin{equation}\label{defbarc0}
\oc(\rho) \ :=\ 
\max_{x\in\partial B_\rho} \max\{\H(x,\overline B_\rho),-\H(x,\R^d\setminus B_\rho)\}\,,
\end{equation}
\begin{equation}\label{defbarc}
\uc(\rho) \ :=\ 
\min_{x\in\partial B_\rho} \min\{\H(x, \overline B_\rho),-\H(x,\R^d\setminus B_\rho)\}\,,
\end{equation}
which are continuous functions of $\rho>0$.
Observe that thanks
to the monotonicity axiom (M), 
the functions $\rho\mapsto\oc(\rho)$
and $\rho\mapsto\uc(\rho)$ are nonincreasing.

In \cite{cmp} (see Theorem \ref{exun} below) it has been proved that the assumptions (M), (T), (C) above are enough to guarantee the existence of a viscosity level set solution of the nonlocal $\H$-curvature flow;
as in \cite{cmp}, a stronger property than (C) will be required in order to have uniqueness for such a solution. Such a property is the  following:
\begin{itemize}
\item[(C')] Uniform continuity: Given $\rr>0$, there exists a modulus of continuity $\omega_\rr$ such that the following holds.
For all $E\in\Reg$, $x\in\partial E$, such that $E$ has both an interior
and exterior ball condition of radius $\rr$ at $x$,
 and for all diffeomorphisms $\Phi:\R^d\to \R^d$   of class {$C^{\regu}$},
with $\Phi(y)=y$ for $|y-x|\ge 1$, we have  
$$
|\H(x,E) - \H(\Phi(x),\Phi(E))|\le \omega_\rr(\|\Phi - \mathrm{Id}\|_{C^{\regu}}).
$$
\end{itemize}

Assumption (B) below will guarantee that the curvature flow starting from
a bounded set remains bounded at all times, yielding existence of global solutions.
\begin{itemize}
\item[(B)] Lower bound on   the  curvature of the balls: 
There exists $K>0$ such that
\begin{equation}\label{lwrbdkappa}
\uc(\rho)\ \geq \ -K \rho   \qquad \text{ for all  } \rho \ge 1.
\end{equation}
\end{itemize}

Such an assumption is useful when dealing with a sequence of curvatures, since it guarantees that the corresponding  geometric evolutions are all defined   on the whole $[0,+\infty)$, 
and hence, in particular, on the same time interval.  Moreover, the assumption (B) could be weakened in several ways with minor changes in the proofs,  by requiring 
for instance  $\uc(\rho)\ \geq \ -K \rho + C$ for some (possibly negative) $C\in\R$. 

The following symmetry condition guarantees that geometric evolutions are preserved passing to the complementary sets, and it will play a role in proving (uniform in) time continuity estimates.  

\begin{itemize}
\item[(S)] Symmetry: 
For all $E\in\Reg$ and for every $x\in\partial E$, it holds $\H(x,E)= -\H(x,\R^d\setminus \overset{\circ}{E})$, where $\overset{\circ}{E}$ denotes the interior of $E$. 
\end{itemize}

\subsection{Notion of viscosity solution}
Let $\H$ be a nonlocal curvature in the sense of Definition \ref{defcu}. 
Given a continuous function $u_0:\R^d\to\R$, constant out of a compact set, we provide the proper notion of solution to the following parabolic
Cauchy problem 
\begin{equation}
\label{levelsetf}
\begin{cases}
 \partial_t u(x,t) + |\D u(x,t)| \H(x,  \{ y: \, u(y,t)\ge u(x,t)\}) \,=\,0 \\ 
 u(\cdot,0) = u_0(\cdot).
\end{cases}
\end{equation}
Here and  in the following, $\D$ and $\D^2$ stand for the spatial gradient and the spatial Hessian matrix, respectively.
Notice that if the superlevel sets of $u$ are not smooth, the meaning of \eqref{levelsetf} is unclear. For this reason, it is necessary to use a definition based on appropriate smooth test functions whose level sets  curvatures  are
well defined. Following \cite{cmp} we will adopt the appropriate framework of  
viscosity solutions.
%
We start by introducing the class of admissible test functions.

Let $\gamma:(0,+\infty)\to\R$  be a nonincreasing  continuous function with $\gamma \ge \oc$, where $\oc(\rho)$ is the function introduced in \eqref{defbarc0}.
Let $\F$ be the family of functions {$f\in C^\infty([0,+ \infty))$} such that $f(0) = f'(0)= f''(0)=0$, $f''(r) >0$ for all $r$ in a neighborhood of $0$, $f$ is constant in $[M, +\infty)$ for some $M>0$ (depending on $f$), and 

\begin{equation}\label{IS}
\lim_{\rho\to 0^+} f'(\rho)\, \gamma(\rho) = 0.
\end{equation}

In  \cite[p. 229]{is} it has been proven that the family $\F$ is not empty.
Note that \eqref{IS} implies
\begin{equation}\label{ISvera}
\lim_{\rho\to 0^+} f'(\rho)\,  \gamma(f^{-1}(\rho)) = 0,
\end{equation}
since $f^{-1}(\rho)>\rho$ for small values of $\rho$ and $\gamma$ is nonincreasing.


With a small abuse of language, we will say that a function $g:\R^d\times A\to \R$, with $A\subseteq [0, +\infty)$, is constant outside a compact set  if for all $t \in A$ and for all $t'\in A\cap [0,t)$ we have
$g(\cdot,t')\equiv C_{t'}$ on $\R^d\setminus \mathcal K_{t}$ for some $C_{t'}\in\R$ and some compact set $\mathcal K_t\subset \R^d$.

\begin{definition}\label{defadmissible}
Let $\hat z=(\hat x,\hat t)\in  \R^d\times (0, +\infty)$ and let $A\subset (0, +\infty)$ be any open bounded interval containing $\hat t$.   We will say that $\ffi\in C^{0}( \R^d\times \overline A)$ is
{\em admissible at the point $\hat z=(\hat x,\hat t)$} if it is 
of class $C^2$ in a neighborhood of $\hat z$, 
if it is constant out of a compact set, 
  and,
in case $\D\ffi(\hat z)=0$, the following holds: 
there exists   $f\in\F$  and { $\omega\in  C^\infty([0,+\infty))$ with
 $\omega(0)= \omega'(0)=0$, $\omega(r)>0$ for $r\neq 0$}
such that
$$
|\ffi(x,t) - \ffi(\hat z) - \ffi_t(\hat z)(t-\hat t)|\le f(|x-\hat x|) + \omega(|t -\hat t|)
$$
for all $(x,t)$ in $\R^d\times A$.
\end{definition}

In the following we will say that a level of a smooth function $\ffi$ is noncritical if $\D \ffi$  does not vanish on  such a level.   
We are now ready to provide the definition of viscosity sub and supersolution as in \cite{cmp}. 
{
\begin{definition}\label{defviscoC2}
An upper semicontinuous function $u:\R^d\times [0,+\infty)\to \R$ (in short $u\in USC(\R^d\times [0,+\infty))$),
constant outside a compact set, is
a viscosity subsolution of the Cauchy problem \eqref{levelsetf} if $u(\cdot,0) \le u_0(\cdot)$ and for all 
$z:= (x,t) \in \R^d\times (0,+\infty)$ and  all $C^\infty$-test functions $\ffi$
such that $\ffi$ is admissible at $z$ and $u-\ffi$ has a
maximum at $z$ (in the domain of definition of $\ffi$)  the following holds:
\begin{itemize}
\item[(i)] If $\D\ffi(z) = 0 $, then $\ffi_t(z) \le 0$;
\item[(ii)]  If the level set $\{\ffi(\cdot,t) = \ffi(z)\}$ is noncritical, then 
$$\ffi_t(z)+ |\D\ffi(z)| \, \H\left(x,\{y: \ffi(y,t)\ge \ffi(z)\}\right)
\le 0.
$$ 
\end{itemize}
A lower semicontinuous function $u$ (in short $u\in LSC(\R^d\times [0,+\infty))$), constant outside a compact set,
is a viscosity supersolution of the Cauchy problem \eqref{levelsetf} if $u(\cdot,0) \ge u_0(\cdot)$ and for all $z=(x,t)\in\R^d\times (0,+\infty)$ 
and  all $C^\infty$-test functions $\ffi$
such that $\ffi$ is admissible at $z$ and $u-\ffi$ has a
minimum at $z$ (in the domain of definition of $\ffi$)  the following holds:
\begin{itemize}
\item[(i)] If $\D\ffi(z) = 0 $, then $\ffi_t(z) \ge 0$;
\item[(ii)] If the level set $\{\ffi(\cdot,t) = \ffi(z)\}$ is noncritical, then 
$$
\ffi_t(z)+|\D\ffi(z)| \, {\H\left(x,\{y: \ffi(y,t) \geq  \ffi(z)\}\right)}
\ge 0.
$$ 
\end{itemize}
Finally, a function $u$ is a viscosity solution of the Cauchy problem \eqref{levelsetf} if its upper semicontinuous envelope is a subsolution and 
 its lower semicontinuous envelope is a supersolution of~\eqref{levelsetf}.
 \end{definition}
 }
{ 
\begin{remark}\label{rm:strettocontatto}
As it is standard in the theory of viscosity solutions, the  maximum in Definition~\ref{defviscoC2} of  subsolutions can be assumed to be strict (and similarly for supersolutions). Indeed,  assume for instance that $u$ is a subsolution, 
{
$u-\ffi$ has a  maximum
}
 at some $(\hat x, \hat t)$, with $\ffi$ as in  Definition~\ref{defviscoC2}. We now replace replace $\ffi$ by
$$
\ffi^s(x,t):= \ffi(x,t)+s  f(|x-\hat x|)+|t-\hat t|^2\,,
$$
where $s>0$ is sufficiently small and $f\in \F$.
Then the maximum of $u-\ffi^s$ at $(\hat x, \hat t)$ is strict and we recover the subsolution  inequality for $\ffi$ by letting $s \to 0$ and using the continuity of $\H$.
\end{remark}
}

\begin{remark}\label{consS}
Property (S) in Subsection \ref{axioms} implies that if $u$ is a continuous viscosity solution to \eqref{levelsetf}, then also $-u$ is a solution. Indeed, assume that $u$ is a subsolution; if $u-\ffi$ attains a maximum at $(x,t)$, then setting $\psi:=-\ffi$,  we have that $-u - \psi$ attains a minimum at $(x,t)$. Moreover, $\{\psi(\cdot,t) \ge \psi(x,t)\} $ is the  complementary set of the interior of $\{\ffi(\cdot,t) \ge \ffi(x,t)\} $. This fact together with (S) implies that $- u$ is a supersolution. Similarly, it follows that, if $u$ is a supersolution, then $-u$ is a subsolution. 
\end{remark}
Throughout the  paper, we will use (with a small abuse of terminology)  the terms subsolutions and supersolutions (omitting the locution ``of the Cauchy problem \eqref{levelsetf}'') also for functions which do not satisfy the corresponding inequalities at time zero. 

\subsection{Existence and uniqueness of a viscosity solution}

The next lemma, proved in \cite{cmp}, establishes a comparison result between viscosity subsolutions and smooth supersolutions, and vice-versa. 

\begin{lemma}\label{lm:classiccomp}
	Let $\H$ be a nonlocal curvature.	
	Let $u\in USC( \RT)$ be a  subsolution  of \eqref{levelsetf}. Let $0\le t_0\le t_1<+\infty$, and let $\ffi\in C^{2}(\R^d \times [t_0,t_1])$ be admissible at all points in the sense of Definition~\ref{defadmissible} and such that $\ffi(\cdot,t_0)\geq u(\cdot,t_0)$, 
	\begin{equation}\label{eqponza}
	\ffi_t(x,t)+|D\ffi(x,t)|\H(x,  \{\ffi(\cdot, t)\geq \ffi(x,t)\})\geq 0
	\end{equation}
	for all $(x,t)\in \R^d \times (t_0,t_1)$, with $|D\ffi(x,t)|\neq 0$,  and  $\ffi_t(x,t)\geq 0$ if $|D\ffi(x,t)|= 0$.  Then, $\ffi\geq u$ in $\R^d \times [t_0,t_1]$. 
	An analogous comparison principle holds between viscosity supersolutions and classical subsolutions. 
\end{lemma}

Let $\overline \gamma:(0,+\infty)\to\R$  be a nonincreasing  $C^1$ function with $\overline\gamma \ge \oc $ (recall \eqref{defbarc0}). 
Let $R_0>0$, $T>0$.
Let $\overline R:[0,T)\to (0,+\infty)$ be a solution to
\begin{equation}\label{superR}
\left\{\begin{array}{l}
\dot{R}(t)=-\overline \gamma(R(t))\\
R(0)=R_0.
\end{array}\right.
\end{equation}
Notice that for $T>0$ small enough such a solution always exists and is uniquely determined. In fact, either the unique solution exists and is strictly positive for all times or
there could exist a maximal interval $[0,T^\ast)$ such that either $\lim_{t\to T^\ast}\overline R(t)=0$ or $\lim_{t\to T^\ast}\overline R(t)=+\infty$. 


The following result establishes that balls evolving with normal velocity given  by 
$\overline\gamma(\rho)$ and $K\rho$ (with $K$ given by property (B))  provide inner and outer barriers, respectively, for any viscosity solution. Its rigorous statement and its standard proof are  slight variants of \cite[Lemma 2.19]{cmp}, 
the difference being that here we deal with curvatures which are not bounded from below by a constant.  

\begin{lemma}\label{compball}
Let $\H$ be a nonlocal curvature satisfying (B).  
Let $\overline \gamma:(0,+\infty)\to\R$  be a nonincreasing $C^1$ function with $\oc\le\overline \gamma$ (with $\oc$  defined in \eqref{defbarc0}). 
Let moreover $R_0>0$ and $t_0\ge 0$. Then we have:

\begin{itemize}
\item[(i)] If  $u$ is a subsolution to \eqref{levelsetf} with $u(\cdot, t_0) \le \lambda+\mu\chi_{B_{R_0}}(\cdot)$ for some $\lambda \in\R,\, \mu >0 $,   then
$$
u(\cdot, t )\le \lambda+\mu\chi_{B_{\underline R(t-t_0)}}(\cdot)\qquad\textrm{for every }t\in [t_0,+\infty),
$$
where $\underline R(t):=R_0 e^{Kt}$ for every $t\ge 0$, with $K$ given in \eqref{lwrbdkappa}.

\item[(ii)] If  $u$ is a supersolution to \eqref{levelsetf} with $u(\cdot, t_0) \ge \lambda+\mu\chi_{B_{R_0}}(\cdot)$ for some $\lambda \in\R,\, \mu <0 $,   then
$$
u(\cdot, t )\ge \lambda+\mu\chi_{B_{\underline R(t-t_0)}}(\cdot)\qquad\textrm{for every }t\in [t_0,+\infty).
$$

\item[(iii)] If  $u$ is a supersolution to \eqref{levelsetf} with
$u(\cdot, t_0)\ge \lambda+\mu\chi_{B_{R_0}}(\cdot)$  for some $\lambda\in\R, \, \mu>0 $, then 
$$
u(\cdot,t)\ge \lambda+\mu\chi_{B_{\overline R(t-t_0)}}(\cdot) \qquad\textrm{for every }t\in [t_0,t_0+T),
$$
where $\overline R$ is a solution to \eqref{superR} in $[0,T)$.

\item[(iv)] If  $u$ is a subsolution to \eqref{levelsetf} with
$u(\cdot, t_0)\le \lambda+\mu\chi_{B_{R_0}}(\cdot)$  for some $\lambda\in\R, \, \mu<0 $, then 
$$
u(\cdot,t)\le \lambda+\mu\chi_{B_{\overline R(t-t_0)}}(\cdot) \qquad\textrm{for every }t\in [t_0,t_0+T).
$$
\end{itemize}
\end{lemma}

\begin{proof}
We only prove (i) and (iii), the proof of (ii) and (iv) being fully analogous. 

We start with the proof of (i); for every $\e>0$ let $\psi^\ep:\R\to [0,+\infty)$ be a smooth nonincreasing function, constant in $(-\infty,0]$, with  support in $(-\infty,R_0+\e]$ , $(\psi^\e)'(0) = (\psi^\e)''(0) = 0$,  $\psi^\e\ge \chi_{(-\infty,R_0]}$ and  $\psi^\e\to \chi_{(-\infty,R_0]}$ pointwise. 
Moreover, introduce the function $\underline{R}^\ep:\R\to\R$ defined by $\underline{R}^\ep(\tau)=(R_0+\ep) e^{K\tau}$. 
For every $\e>0$ we set
\begin{equation*}
\ffi^\e(x,t):=\lambda+\mu\psi^\e (|x|+R_{0}+\ep- \underline{R}^\ep(t-t_0)),\qquad x\in\R^d, t\ge t_0.
\end{equation*} 
One can always choose $\psi^\e$ flat enough wherever $(\psi^\e)'=0$, so that the functions $\ffi^\e$ are  admissible at all points.
By assumption
$$
\ffi^\e(x,t_0)=\lambda+\mu\psi^\e (|x|)\geq  \lambda+\mu\chi_{B_{R_0}}(x)\ge u_0(x)\,.
$$  


%
%

By a direct computation we will briefly show  that $\ffi^\ep$ satisfies the assumptions of Lemma \ref{lm:classiccomp} above, i.e., it is a smooth supersolution to \eqref{levelsetf}. On the one hand, recalling that $\ffi^\e$ is constant whenever the argument of $\psi^\e$ is nonpositive, we easily get that  if $ \D\ffi^\ep(x,t) =0$, then 
$\ffi^\ep_t(x,t)=0$. 
On the other hand, in the  noncritical case
all the superlevels of $\ffi^\e$ expand with normal velocity equal to  $ K \underline{R}^\ep(t-t_0)$, so that
\begin{equation*}
\begin{aligned}
\ffi^\ep_t(x,t)=  &|\D\ffi^\ep(x,t)| K \underline{R}^\ep(t-t_0) \ge  -|\D\ffi^\ep(x,t)| \H(x,\overline{B}_{ \underline{R}^\ep(t-t_0)})\\
\ge &-|\D\ffi^\ep(x,t)| \H(x,\{\ffi^\ep(\cdot,t)\ge\ffi^\ep (x,t)\})\,,
\end{aligned}
\end{equation*}
where the first inequality is a consequence of (B) and the second one follows noticing that, by the very definition of $\ffi^\e$,
$$
\{\ffi^\ep(\cdot,t)\ge\ffi^\ep (x,t)\}\subset B_{ \underline{R}^\ep(t-t_0)}
$$
and using property (M).
By Lemma \ref{lm:classiccomp}, we deduce that $u\le \ffi^\ep$, which, sending $\ep\to 0$, yields $u\le \lambda+\mu\chi_{B_{\underline{R}(t-t_0)}}$.

Let us pass to the proof of (iii).  For every $\e>0$ let now $\psi^\ep:[0,+\infty)\to [0,+\infty)$ be a smooth nonincreasing function, with  support in $(-\infty,R_0-\e]$ , $(\psi^\e)'(0) = (\psi^\e)''(0) = 0$,  $\psi^\e\le \chi_{[0,R_0]}$ and  $\psi^\e\to \chi_{[0,R_0]}$ pointwise. 
Moreover, consider the Cauchy problem
\begin{equation}\label{superRep}
\left\{\begin{array}{l}
\dot{R}(t)=-\overline \gamma(R(t))\\
R(0)=R_0-\ep.
\end{array}\right.
\end{equation}
Since $\overline\gamma\in C^1$, for $\e>0$ small enough the (unique) 
solutions $\overline{R}^\ep$
to the problem \eqref{superRep} 
are  all defined and $C^2$ regular on some interval $[0,T)$,
with $T>0$ independent of $\e$, and uniformly converge to the solution 
$\overline{R}$ to \eqref{superR}. 

Now, setting, for every $\e>0$,
\begin{equation*}
\ffi^\e(x,t):=\lambda+\mu\psi^\e (|x|+R_{0}-\ep- \overline{R}^\ep(t-t_0)),\qquad x\in\R^d, t\ge t_0,
\end{equation*} 
and arguing as in the proof of (i), one can easily show the claim.
\end{proof}

While Lemma \ref{lm:classiccomp} provides a comparison principle between viscosity solutions and regular evolving sets, the following result establishes the comparison principle between pairs of viscosity solutions. 

\begin{theorem}[Comparison Principle] \label{th:CP}
Assume that $\H$ is a nonlocal curvature satisfying (C') and (B).  
Let $u\in USC( \RT)$ and $v\in LSC(\RT)$, both constant (spatially)
out of a compact set, be a subsolution and a supersolution of \eqref{levelsetf}, respectively. If $u(\cdot,0) \le v(\cdot,0)$,  
then $u\le v$ in $\RT$.
\end{theorem}

The  existence and uniqueness result for viscosity solutions to \eqref{levelsetf} is provided by the following theorem.

\begin{theorem}[Existence and uniqueness]\label{exun}
Assume that $\H$ is a nonlocal curvature satisfying (B).
Let $u_0\in C(\R^d)$ be a uniformly continuous function with $u_0=C_0$ in $\R^d\setminus B_{R_0}$, for some $C_0,R_0\in\R$ with $R_0>0$. 

There exists a  viscosity solution $u:\R^d\times [0,+\infty)$ to \eqref{levelsetf}. Moreover, any viscosity solution satisfies $u(\cdot,t)=C_0$ in    $\R^d\setminus B_{R(t)}$, where $R(t):= R_0 e^{Kt}$ with $K$ introduced in \eqref{lwrbdkappa}.

Finally, if $\H$ satisfies (C'), then the viscosity solution  to \eqref{levelsetf} is unique. 
\end{theorem}

Theorems \ref{th:CP} and  \ref{exun} have been established in \cite{cmp} under minor irrelevant differences in the assumptions and in the claims. In fact, in \cite{cmp} the class of test functions has been defined for $\gamma=\oc$. However, replacing 
 $\oc$ with a larger $\gamma$ does not affect any step in the proof.  
 Moreover,  
the proof of  the existence in Theorem \ref{exun} follows along the lines of the results \cite[Subsection 2.6]{cmp};  there, the existence of a viscosity solution in any given compact interval $[0,T]$ is provided, under the assumption (B') that the curvature of balls $B_R$ is bounded from below by a constant $-K$ independent of $R$. However, the results in \cite{cmp} can be easily extended to obtain global solutions also replacing (B') with our weaker assumption (B) and using Lemma \ref{compball} in place  of \cite[Lemma 2.19]{cmp}.

\begin{remark}
The notion of viscosity solutions apparently depends   on the choice of the function $\gamma$ in Definition \ref{defadmissible}. In fact, such a dependence is fictitious, since  if $\gamma_2\ge \gamma_1$, then the class of admissible test functions corresponding to $\gamma_2$ are also admissible replacing $\gamma_2$ with $\gamma_1$. Therefore, the unique solution provided by Theorem \ref{exun} does not depend on the specific choice of $\gamma$.
\end{remark}


Given a continuous function $v:\R^d \times A\to \R$, with $A\subseteq [0,+\infty)$,  we will say that a function $\osp:[0,+\infty)\to[0,+\infty)$ is a spatial  modulus of continuity of $v$ if it is strictly increasing and continuous in $[0,R)$ for some $R>0$, it  satisfies $\osp(0)=0$ and 
$$
|v(x_1,t)-v(x_2,t)|\le \osp(|x_1-x_2|) \qquad \text{ for all } x_1,\,x_2\in\R^d, \, t\in A.
$$ 
The time moduli of continuity are defined analogously and denoted by $\oti$.

\begin{proposition}[Uniform continuity]\label{ucest}
Let $\osp$ be a  spatial modulus of continuity    and let $\gamma:(0,+\infty)\to\R$  be a nonincreasing $C^1$ function.

Then, there exists a time modulus of continuity $\oti$ such that, 
for all nonlocal curvatures $\H$ satisfying (B), (C'), (S), and  such that $\oc\le\gamma$ (with $\oc$  defined in \eqref{defbarc0}) and  for any initial datum $u_0\in C(\R^d)$   
constant out of a compact set and  
with spatial modulus of continuity  $\osp$, we have that 
 $\osp$ and $\oti$ are spatial and time moduli of continuity, respectively, of the  viscosity solution  to \eqref{levelsetf} with initial datum $u_0$.
\end{proposition}
\begin{proof}
The fact that $\osp$ is a spatial modulus of continuity for $u$ is a well known consequence of the comparison principle and the invariance by translations of the curvature $\H$. For the reader's convenience we provide such a proof. 
Let $\eta \in  \R^d$ and set $v^\eta_0(x):=u_0(x+\eta)+\osp(|\eta|)$ for every $x\in\R^d$. Clearly, $v_0$ is constant outside a compact set, and, since $\osp$ is a modulus of continuity of $u_0$, we have
$$
u_0(x) \le u_0(x+\eta)+\osp(|\eta|)=v^\eta_0(x)\qquad\textrm{for every }x\in\R^d.
$$  
By Theorem \ref{exun} and by property (T), the function $v^\eta:\R^d\times [0,+\infty)\to\R$ defined by $v^\eta(x,t):=u(x+\eta,t)+\osp(|\eta|)$ is the viscosity solution of \eqref{levelsetf} with initial datum $v_0^\eta$.
By Theorem \ref{th:CP} we have 
$$
u(x,t) \le u(x+\eta,t)+\osp(|\eta|)\quad\textrm{for every }x,\eta\in\R^d,t\in [0,+\infty).
$$

We now prove the existence of the modulus of continuity $\oti$ for $u$, i.e., we show that for every $\ep>0$ there exists $\tau_\ep>0$ such that
\begin{equation}\label{uctime}
u(x,t_0) -\ep\le u(x,t)\le u(x,t_0)+\ep \quad\textrm{for every }x\in\R^d\,,\, 
0\le t_0\le t\le t_0+\tau_\ep.
\end{equation}
 We start by proving the first inequality in \eqref{uctime}. 
Let $\ep>0$ and set $R^\ep_0:=(\osp)^{-1}(\ep)$. Then, for every $x\in\R^d$
$$
B_{R^\ep_0}(x)\subseteq\{y\,:\,u(y,t_0)>u(x,t_0)-\ep\}.
$$
Therefore, $(u(x,t_0)-\ep)\chi_{B_{R^\ep_0}(x)}(\cdot)\le u(\cdot, t_0)$. 
Let $\overline R(t)$ be the solution to \eqref{superR} with $\overline\gamma=\gamma$ and $R_0=R^\ep_0$;
notice that
$\overline R(t)\ge \frac{R^\ep_0}{2}$ for every $t_0\le t\le t_0+\tau_\ep$ where $\tau_\ep:=\frac{R^\ep_0}{2\gamma(R^\ep_0/2)}$ if $\gamma(R^\ep_0/2)>0$ and $\tau_\ep=1$ otherwise.

By Lemma \ref{compball}, we deduce that
$$
u(x,t)\ge (u(x,t_0)-\ep)\chi_{B_{\frac{R^\ep_0}{2}}(x)}(x)=u(x,t_0)-\ep \qquad\textrm{for every }t_0\le t\le t_0+\tau_\ep.
$$
 The second inequality in \eqref{uctime} is a direct consequence of the first one, applied to the function $-u$, which, in view of Remark \ref{consS}, is a viscosity solution to \eqref{levelsetf} with initial datum $-u_0$.


\end{proof}

\section{Convergence of nonlocal curvature flows}\label{general}

In this section we consider a sequence $\{\H^n\}_{n\in\N}$ of nonlocal curvatures and introduce a notion of convergence of $\H^n$ to some limit curvature $\H^\infty$. 
In order to have uniform continuity estimates of the corresponding level set solutions, we enforce uniform bound on the velocity of blowing up and blowing down evolving balls;
more precisely, let $\uc^n$ and $\oc^n$
be defined as in  \eqref{defbarc0} and \eqref{defbarc}, with $\H$ replaced by $\H^n$, we set 
\begin{equation}\label{defucinfty}
\uc^{\inf} (\rho):= \inf_{n\in\N}  \uc^n (\rho), \qquad \oc^{\sup} (\rho):= \sup_{n\in\N}  \oc^n (\rho). 
\end{equation}
We assume that

\begin{itemize}
\item[(UB)] There exists $K\ge 0$ such that
$\uc^{\inf} (\rho) \ge -K\rho$ for all $\rho>1$, and  
$\oc^{\sup} (\rho)<+\infty$ for all $\rho>0$.  
\end{itemize}

\begin{definition} \label{curvconver}
Let $\{\H^n\}_{n\in\N}$ be a sequence of nonlocal curvatures and let $\H^\infty$ be a nonlocal curvature. We say that $\{\H^n\}_{n\in\N}$ converges to $\H^{\infty}$ and we write $\H^n\to \H^\infty$ whenever for every $\{E_n\}_{n\in\N}\subset\Reg$, $E\in\Reg$ with $E_n\to E$ in $\Reg$, and for every $x\in\partial E\cap\partial E_n$, it holds $\H^n(x,E_n)\to\H^{\infty}(x,E)$.
\end{definition}
Notice that, if $\{\H^n\}_{n\in\N}$ satisfies (UB) and $\H^n\to\H^{\infty}$, then $\H^{\infty}$ satisfies (B). Moreover, if $\H^n$ satisfy (S) for every $n\in\N$ and $\H^n\to\H^{\infty}$, then $\H^{\infty}$ satisfies (S).

\begin{theorem}\label{genthm}
Let $\{\H^n\}_{n\in\N}$ be a sequence of nonlocal curvatures  satisfying (C'), (S) and (UB) such that $\H^n\to\H^\infty$ for some  nonlocal curvature $\H^{\infty}$ satisfying  (C').

Let $u_0\in C(\R^d)$ be a uniformly continuous function with $u_0=C_0$ in $\R^d\setminus B_{R_0}$, for some $C_0,R_0\in\R$ with $R_0>0$. 
For every $n\in\N$ let $u^n$ be the viscosity solution to \eqref{levelsetf} with $\H$ replaced by $\H^n$.
Then $u^n\to u^\infty$ where $u^\infty:\R^d\times [0,+\infty)\to\R$ is the (unique) viscosity solution to \eqref{levelsetf} with $\H$ replaced by $\H^\infty$.
\end{theorem}
\begin{proof}
Notice that, under the assumption (UB), the function $\oc^{\sup}$ is finite and nonincreasing; therefore, there exists a nonincreasing $C^1$ function $\gamma:(0,+\infty)\to\R$ with $\gamma\ge \oc^{\sup}$.

Let $\overline R(t)$ be a solution to \eqref{superR} with $\overline\gamma=\gamma$, and  set $\underline R(t):=R_0 e^{K t}$ for every $t\ge 0$, with $K$ given by (UB).

By Theorem \ref{exun}, given $T>0$, the functions $u^n(x,t)\equiv C_0$ for all $0\le t\le T$ and for all $x\in \R^d \setminus B_{\underline R(T)}$.  
By Proposition \ref{ucest} and by Ascoli-Arzel\'a Theorem, we have that, up to a (not- relabeled) subsequence, $u^n\to u^\infty$ locally uniformly on $\R^d\times [0,+\infty)$, for some function $u^\infty$ with $u^\infty (\cdot,t)\equiv C_0$ in $\R^d \setminus B_{\underline R(t)}$ for every $t\ge 0$, and such that $u^\infty$ has the same spatial modulus of continuity $\osp$ of $u_0$ and time modulus of continuity $\oti$ given by Proposition \ref{ucest}.

It remains to show that $u^\infty$ is a viscosity solution to \eqref{levelsetf} with $\H$ replaced by $\H^\infty$.
We only prove that $u$ is a subsolution, since the proof that $u$ is a supersolution is identical, and we will verify Definition \ref{defviscoC2} employing test functions $\ffi$ which are admissible, according to Definition \ref{defadmissible}, with the choice of $\gamma$ done at the beginning of this proof.

Let $\overline z=(\overline x,\overline t)\in \R^d\times (0,+\infty)$ and let $\ffi$ be a test function which is admissible at $\overline z$, such that $u^\infty-\ffi$ has a maximum at $\overline z$. 

First we consider the case that the level $\{\ffi(\cdot,\overline t)= \ffi (\overline z)\}$ is noncritical.  In view of Remark \ref{rm:strettocontatto}, we can assume that the maximum at $\overline z$ is strict. Since $u^n\to u^\infty$ locally uniformly, for $n$ large enough, $u^n-\ffi$ admits a maximum at a point $\overline z^n=(\overline x^n,\overline t^n)$ for some $\overline z^n\to \overline z$.
Moreover, for $n$ large enough, the levels $\{\ffi(\cdot, \overline t^n)=\ffi(\overline z^n)\}$ are noncritical. Since $u^n$ are subsolutions, we have
$$
\partial_t \ffi(\overline z^n) + |\D \ffi(\overline z^n)| \H^n(\overline x_n,  \{\ffi(\cdot,\overline t^n)\ge \ffi(\overline z^n)\}) \le 0,
$$
which, letting $n\to +\infty$, in view of properties (C) and (T), yields (ii) of Definition \ref{defviscoC2}.

Assume now that $\D\ffi(\overline z)=0$; then
$$
|\ffi(x,t) - \ffi(\overline z) - \ffi_t(\overline z)(t-\hat t)|\le f(|x-\overline x|) + \omega(|t -\overline t|),
$$
with the functions $f$ and $\omega$ as in Definition \ref{defadmissible}.
Let us consider the admissible test function
$$
\psi(x,t):=\ffi_t(\overline z)(t-\overline t)+2f(|x-\overline x|)+2\omega(|t-\overline t|).
$$
It is easy to see that $u^n-\psi$ admits a strict maximum at some $\overline z^n=(\overline x^n,\overline t^n)$ with $\overline z^n\to\overline z$. If $\D\psi (\overline z^n)= 0$, then
\begin{equation}\label{flat}
0\ge \partial_t \psi(\overline z^n)=\ffi_t(\overline z)+2 \,\mathrm{sgn}(\overline t_n-\overline t)\,\omega'(|\overline t_n-\overline t|).
\end{equation}
If $\D\psi (\overline z^n)\neq 0$, then, since $\psi$ is radial, the level $\{\psi(\cdot,\overline t^n)=\psi(\overline z^n) \}$ is noncritical, and  
\begin{equation}\label{nonflat}
\begin{aligned}
0\ge& \partial_t \psi(\overline z^n)+|\D\psi (\overline z^n)|\H^n(\overline x^n, \{\psi(\cdot,\overline t^n)\ge \psi(\overline z^n) \})\\
=&\ffi_t(\overline z)+2 \,\mathrm{sgn}(\overline t_n-\overline t)\,\omega'(|\overline t_n-\overline t|)+2|f'(|\overline x^n-\overline x|)|\H^n(\overline x^n,\chi_{\R^d\setminus B_{|\overline x-\overline x^n|}(\overline x)})\\
\ge & \ffi_t(\overline z)+2 \,\mathrm{sgn}(\overline t_n-\overline t)\,\omega'(|\overline t_n-\overline t|)-2|f'(|\overline x^n-\overline x|)| \gamma(|\overline x-\overline x^n|).
\end{aligned}
\end{equation}
By \eqref{flat} and \eqref{nonflat}, letting $n\to +\infty$, recalling also \eqref{ISvera} and the fact that $\omega'(0)=0$, we have that $u^\infty$ satisfies condition (i) of  Definition \ref{defviscoC2}.
\end{proof}

\section{Convergence of  fractional mean curvature flows} \label{fraccurv}
Here we apply the general theory of Section \ref{general} to study the limit cases of $s$-fractional mean curvature flows as $s\to 0$ and $s\to 1$. 

\subsection{Fractional perimeters and curvatures}
Let $s\in (0,1)$. For every $E\in\Reg$
and for every  $x\in\partial E$, the \emph{$s$-fractional curvature} of $E$ at $x$ is defined \cite{i} by
\begin{equation}\label{scurv}
\hh^s(x, E)=   \lim_{r\to 0^+}  \int_{\Rd\setminus B_r(x)} \frac{\chi_{\R^d\setminus E}(y)-\chi_{E}(y)}{|x-y|^{d+s}}\ud y. 
\end{equation}

It is well known  that the curvature $\H^s$ is the first variation of the $s$-fractional perimeter $\P^s$
defined by
\begin{equation}\label{sper}
\P^s(E):=\int_{E}\int_{\Rd\setminus E}\frac{1}{|x-y|^{d+s}}\ud y\ud x.
\end{equation}

The following definition introduces the notion of \emph{$0$-fractional curvature}.

\begin{definition}[The zero curvature] 
Let $E\in\Reg$. 
We define the $0$-fractional curvature as 
\begin{equation}\label{zerocurv} \hh^0(x,E):=
\begin{cases} \displaystyle \lim_{R\to +\infty}\lim_{r\to 0^+}\int_{B_R(x)\setminus B_r(x)} \frac{\chi_{\Rd\setminus E}(y)-\chi_{E}(y)}{|x-y|^d}\ud y-d\omega_{d} \log R\,,
& E\subset\subset \R^d\,,
\\  \displaystyle \lim_{R\to +\infty}\lim_{r\to 0^+}\int_{B_R(x)\setminus B_r(x)} \frac{\chi_{\Rd\setminus E}(y)-\chi_{E}(y)}{|x-y|^d}\ud y+d\omega_d \log R\,, & \R^d\setminus E\subset\subset \R^d\,.\end{cases}\end{equation} 
\end{definition}

Notice that for every compact set $E\in \Reg$ we have $\hh^0(x,E) = \hh^0(x,\R^d\setminus \overset{\circ}{E})$ and, for all
$R>\dia(E)$,
\begin{equation}\label{noRtoin}
\hh^0(x,E)=\lim_{r\to 0^+}\int_{B_R(x)\setminus B_r(x)} \frac{\chi_{\Rd\setminus E}(y)-\chi_{E}(y)}{|x-y|^d}\ud y-d\omega_{d} \log R.
\end{equation}

Observe also that   $\hh^0(x, E)$ coincides with $L_\Delta\chi_E(x)$, up to a dimensional constant, where $L_\Delta$ is the logarithmic Laplacian  introduced in \cite{cw}. 

\begin{remark}
One can prove that, for $E$ compact,  $\hh^0(x,E)$ is the first variation of the $0$-fractional perimeter $\P^0$ introduced in \cite{dnp}, 
 and defined for all measurable sets $E$ with finite volume as
\begin{equation}\label{0per}
\P^0(E):=   \int_{E}\left[\int_{B_R(x) \setminus E}\frac{1}{|x-y|^{d}}\ud y-\int_{E\setminus B_R(x)}\frac{1}{|x-y|^{d}}\ud y-d\omega_d\log R\right]\ud x.
\end{equation} It is proved in \cite{dnp} that \eqref{0per} is independent of $R$. 
We also notice that for all $s\in[0,1)$ the $s$-fractional curvatures introduced above are well defined for  all sets with compact boundary  of class $C^{1,1}$. 
This is a consequence of the fact that if $B^\pm$ are interior and exterior tangent balls to $\partial E$ at $x$, then, due to the symmetry of the kernel, the integral contributions of $\hh^s(x,E)$ on $B^\pm$ cancel each other, while outside such balls the kernel is integrable.   This well known fact is the content of the two lemmas below.

\end{remark}


\begin{lemma}\label{inou}
Let $\delta>0$ and let $\overline s\in [0,1)$. Then, there exists an increasing continuous function $\omega_{\delta}^{\overline s}:[0,+\infty)\to \R$ with $\omega_{\delta}^{\overline s}(0)=0$ such that for every $s\in [0,\overline s]$ and for every $\eta\ge 0$ 
\begin{equation}\label{finally}
\int_{B_\eta\setminus (B_\delta(\delta e_d)\cup B_{\delta}(-\delta e_d))}\frac{1}{|y|^{d+s}}\ud y\le \omega^{\overline s}_\delta (\eta)\,,
\end{equation}
where $e_d$ denotes the $d$-th vector of the canonical basis of $\R^d$. 
\end{lemma}

\begin{lemma}\label{inou2}
Let $E\in\Reg$ and let  $\delta>0$ be such that $B_\delta(-\delta e_d )\subseteq E$ and $B_\delta( \delta e_d)\subseteq\Rd\setminus E$. Then, for every $s\in [0,1)$ and for every $\eta>0$, it holds
\begin{equation}\label{uni}
\lim_{r\to 0^+}  \int_{B_{\eta} \setminus B_r} \frac{\chi_{\R^d\setminus   E}(y)-\chi_{ E}(y)}{| y|^{d+s}}\ud y= \int_{B_{\eta} \setminus (B_\delta(-\delta e_d)\cup  B_\delta(\delta e_d)) } \frac{\chi_{\R^d\setminus E}(y)-\chi_{ E}(y)}{| y|^{d+s}}\ud y.
\end{equation} 
\end{lemma}
\begin{proposition}\label{properties}
For every $s\in[0,1)$, the functionals $\hh^s$ satisfy the properties (M), (T), (C'), (B), and (S) in Subsection \ref{axioms}.
\end{proposition}

\begin{proof}
The properties (M), (T) and (S) are easy consequences of the very definitions of the fractional curvatures. 
Concerning property (B), it is trivially satisfied by $\hh^s$ for positive $s$, since in this case the fractional curvature of any ball is always positive. For $s=0$, we first notice that
$\hh^0$ satisfies the following scaling property: for every compact set $E\in\Reg$ we have
\begin{equation}\label{resc} 
\hh^0(\lambda x, \lambda E)=\hh^0(x, E )-d\omega_{d} \log \lambda\qquad\textrm{for all }\lambda>0, x\in\partial E;
\end{equation}
therefore, we deduce that 
  \begin{equation}\label{ball}
  \hh^0(\rho x,\overline B_\rho)= \hh^0(x,\overline B_1)-d\omega_{d} \log \rho\qquad\textrm{for all }\rho>0, x\in\partial B_1,
  \end{equation}  
which yields property (B) also for $s=0$. 

It remains to show that property (C') holds. We prove it only for $s=0$, being the case $s\in (0,1)$ fully analogous and claimed in \cite[Remark 5.2]{cmp}.
In view of property (S) it is enough to prove the claim only for compact sets $E\in\Reg$.
Let $\rr>0$, and let $E\in\Reg$, $x\in\partial E$ be such that $E$ has both an interior
and exterior ball condition of radius $\rr$ at $x$. We denote by $B_\rr^{\inn}$ and $B_\rr^{\ou}$ the interior and exterior tangent balls  to $\partial E$ at $x$, respectively.
Let $\Phi:\R^d\to \R^d$ be a diffeomorphism of class {$C^{\regu}$},
with $\Phi(y)=y$ for $|y-x|\ge 1$.
We set $\widetilde E:=\Phi(E)$ and we notice that for $\|\Phi - \mathrm{Id}\|_{C^{\regu}}$ small enough $\widetilde E$ satisfies interior and exterior ball condition at $\widetilde x:=\Phi (x)$ with radius $\frac \rr 2$. We denote such tangent balls with   $\widetilde B_{\frac{\rr}{2}}^{\inn}$ and $\widetilde B_{\frac \rr 2}^{\ou}$.

Let $R> \max\{ \dia(E),1 \} $. Notice that, if $\|\Phi - \mathrm{Id}\|_{C^{\regu}}$ is small enough, then, also $\dia(\widetilde E)<R$. 
Moreover, by applying Lemma \ref{inou} and Lemma \ref{inou2} with $\delta=\frac{\rr}{2}$, for every $\eta>0$ we have that
\begin{equation}\label{menouno}
\begin{aligned}
\left|\lim_{r\to 0^+}  \int_{B_{\eta}(\widetilde x) \setminus B_r(\widetilde x)} \frac{\chi_{\R^d\setminus  \widetilde E}(y)-\chi_{ \widetilde E}(y)}{|\widetilde x- y|^{d}}\ud y\right|&\\
+\left|\lim_{r\to 0^+}
\int_{B_{\eta}(x) \setminus B_r(x)} \frac{\chi_{\R^d\setminus   E}(y)-\chi_{ E}(y)}{|x- y|^{d}}\ud y\right|&\le 4\omega^0_{\frac\rr 2}(\eta).
\end{aligned}
\end{equation}
Let $\ep>0$ and let $\eta_\ep$ be such that $4\omega^0_{\frac\rr 2}(\eta_\ep)\le \frac\ep 3$ and $\omega^0_{\frac\rr 2}(2\eta_\ep)\le \frac\ep 3$.

By \eqref{menouno}, we have
\begin{equation} \label{zero}
 \begin{aligned}
&|\hh^0(x,E) - \hh^0(\Phi(x),\Phi(E))|
\\
\le&
\left|\int_{B_R(x)\setminus B_{\eta_\ep}(x)}\frac{\chi_{\Rd\setminus E}(y)-\chi_{E}(y)}{|x-y|^d}\ud y - \int_{B_R(\widetilde x)\setminus B_{\eta_\ep}(\widetilde x)}\frac{\chi_{\Rd\setminus \widetilde E}(y)-\chi_{\widetilde E}(y)}{|\widetilde x-y|^d}\ud y\right|+\frac{\ep}{3}\\
\le & \left|\int_{B_R(x)\setminus B_{\eta_\ep}(x)}\frac{\chi_{\Rd\setminus E}(y)-\chi_{E}(y)}{|x-y|^d}\ud y - \int_{\Phi(B_R(x)\setminus B_{\eta_\ep}(x))}\frac{\chi_{\Rd\setminus \widetilde E}(y)-\chi_{\widetilde E}(y)}{|\widetilde x-y|^d}\ud y\right|\\
&
+\left|\int_{\Phi(B_R(x)\setminus B_{\eta_\ep}(x))}\frac{\chi_{\Rd\setminus \widetilde E}(y)-\chi_{\widetilde E}(y)}{|\widetilde x-y|^d}\ud y- \int_{B_R(\widetilde x)\setminus B_{\eta_\ep}(\widetilde x)}\frac{\chi_{\Rd\setminus \widetilde E}(y)-\chi_{\widetilde E}(y)}{|\widetilde x-y|^d}\ud y\right| +\frac \ep 3,
\end{aligned}
\end{equation}
where in the last step we have used triangular inequality.

On the one hand, using that $\Phi=\mathrm{Id}$ outside from $B_1(x)$, we have
\begin{equation}\label{uno}
\begin{aligned}
&\left|\int_{B_R(x)\setminus B_{\eta_\ep}(x)}\frac{\chi_{\Rd\setminus E}(y)-\chi_{E}(y)}{|x-y|^d}\ud y - \int_{\Phi(B_R(x)\setminus B_{\eta_\ep}(x))}\frac{\chi_{\Rd\setminus \widetilde E}(y)-\chi_{\widetilde E}(y)}{|\widetilde x-y|^d}\ud y\right|\\
=& \left|\int_{B_1(x)\setminus B_{\eta_\ep}(x)}\frac{\chi_{\Rd\setminus E}(y)-\chi_{E}(y)}{|x-y|^d}\ud y - \int_{\Phi(B_1(x)\setminus B_{\eta_\ep}(x))}\frac{\chi_{\Rd\setminus \widetilde E}(y)-\chi_{\widetilde E}(y)}{|\widetilde x-y|^d}\ud y\right|\\ 
\le& \frac{\omega_d}{\eta_\ep^d}\|\Phi-\mathrm{Id}\|_{C^1};
\end{aligned}
\end{equation}
on the other hand, for $ \|\Phi-\mathrm{Id}\|_{C^1}$ small enough, 
$$
\Big(\Phi(B_R(x)\setminus B_{\eta_\ep}(x))\Big)\Delta\Big( B_R(\widetilde x)\setminus B_{\eta_\ep}(\widetilde x)\Big)\subset B_{2\eta_\ep}(\widetilde x)\setminus (\widetilde B_{\frac \rr 2}^{\inn}\cup \widetilde B_{\frac \rr 2}^{\ou}),
$$
so that, by Lemma \ref{inou}, we have
\begin{equation}\label{due}
\begin{aligned}
&\left|\int_{\Phi(B_R(x)\setminus B_{\eta_\ep}(x))}\frac{\chi_{\Rd\setminus \widetilde E}(y)-\chi_{\widetilde E}(y)}{|\widetilde x-y|^d}\ud y- \int_{B_R(\widetilde x)\setminus B_{\eta_\ep}(\widetilde x)}\frac{\chi_{\Rd\setminus \widetilde E}(y)-\chi_{\widetilde E}(y)}{|\widetilde x-y|^d}\ud y\right|\\
\le&\int_{B_{2\eta_\ep}(\widetilde x)\setminus (\widetilde B_{\frac \rr 2}^{\inn}\cup \widetilde B_{\frac \rr 2}^{\ou})}\frac{1}{|\widetilde x-y|^d}\ud y\\
\le&\omega^0_{\frac{\rr}{2}}(2\eta_\ep)\le\frac{\ep}{3}.
\end{aligned}
\end{equation}
In view of \eqref{zero}, \eqref{uno} and \eqref{due}, there exists $\tau=\tau_\ep>0$ such that
$$
|\hh^0(x,E) - \hh^0(\Phi(x),\Phi(E))|\le \ep\qquad\textrm{ for }\|\Phi-\mathrm{Id}\|_{C^1}\le \tau.
$$
\end{proof}
Existence and uniqueness of  $s$-fractional mean curvature flows for $s\in (0,1)$ has been established in \cite{i} (see also \cite{cmp}). The following result, which is a direct consequence of Theorem \ref{exun} and of Proposition \ref{properties}, establishes existence and uniqueness also for the $0$-fractional mean curvature flow.

\begin{theorem}\label{exunspos}
For every $s\in [0,1)$ and for every uniformly continuous function $u_0\in C(\R^d)$ constant outside a compact set, there exists a unique viscosity solution to \eqref{levelsetf} with $\H$ replaced by $\hh^s$.
\end{theorem}


\subsection{Convergence of the $s$-fractional mean curvature flow as $s\to 0^+$}

\begin{theorem}\label{convergence0}
Let $\{s_n\}_{n\in\N}\subset(0,1)$ be such that $s_n\to 0$ as $n\to+\infty$. Let $\{E_n\}_{n\in\N}\subset\Reg$ be such that $E_n\to E$ in $\Reg$ for some $E\in\Reg$.
For every $x\in\partial E\cap \partial E_n$ it holds 
\begin{equation}\label{order1}
\lim_{n\to +\infty} \hh^{s_n}(x, E_n)-\frac{d\omega_d}{s_n} = \hh^0(x, E),
\end{equation}
 where $\hh^0$ is defined in \eqref{zerocurv}. 
In particular,
\begin{equation}\label{order0}
\lim_{n\to +\infty} s_n\hh^{s_n}(x,E_n)=d\omega_d.
\end{equation}
\end{theorem}
\begin{proof} 
 We start by proving \eqref{order0} (that for $E_n\equiv E$ has  already been proved in \cite{v}). 

Let  $x\in \partial E_n\cap \partial E$ for all $n\in\N$.
First, notice that all the curvatures we are dealing with satisfy assumption (S) and are invariant by rotations (and translations). In particular,  we can assume  without loss of generality that  $E_n$ and $E$ are compact, that $x=0$ 
 and  $\nu_E(0)= \nu_{E_n} (0) =e_d$, where for all $F\in\Reg$ and $y\in\partial F$ we denote by $\nu_F(y)$  the outer  normal to $\partial F$ at $y$.

Setting $\eta:= 2 \dia ( E)$,
we have $E, \, E_n \subset B_\eta$ for $n$ large enough. 
We get 
\begin{multline}\label{fuori} 
\hh^{s_n}(0,E_n)= \lim_{r\to 0^+}  \int_{B_{\eta} \setminus B_r} \frac{\chi_{\R^d\setminus   E_n}(y)-\chi_{E_n}(y)}{| y|^{d+s_n}}\ud y+   \int_{\Rd\setminus B_{\eta}} \frac{1}{| y|^{d+s_n}}\ud y\\ =  \lim_{r\to 0^+}  \int_{B_{\eta} \setminus B_r} \frac{\chi_{\R^d\setminus E_n}(y)-\chi_{E_n}(y)}{| y|^{d+s_n}}\ud y+\frac{d\omega_d }{s_n}\eta^{-s_n}.\end{multline}  
Since $ E_n \to E$ in $\Reg$,   we deduce that $E_n$ and $E$   satisfy a uniform  interior and exterior ball condition. More precisely,  there exists  $\delta>0$ such that $B_\delta(-\delta e_d )\subseteq E_n$ and $B_\delta( \delta e_d)\subseteq\Rd\setminus E_n$ and the same for $E$. 
%

In view of \eqref{fuori} and \eqref{uni}, we deduce that
\begin{equation}\label{fine0}
\begin{aligned}
&\lim_{n\to+\infty}s_n\hh^{s_n}(x,E_n)\\
=&\lim_{n\to+\infty}s_n\int_{B_{\eta} \setminus (B_\delta(-\delta e_d)\cup  B_\delta(\delta e_d)) } \frac{\chi_{\R^d\setminus E_n}(y)-\chi_{ E_n}(y)}{| y|^{d+s_n}}\ud y+ \lim_{n\to +\infty}{d\omega_d }\eta^{-s_n}\\
=&d\omega_d,
\end{aligned}
\end{equation}
where the last equality follows from \eqref{finally},
noticing that $s_n\in (0,\frac 1 2]$ for $n$ large enough.

Let us pass to the proof of \eqref{order1}.
By arguing as in \eqref{fuori}, we have 
\begin{eqnarray}\nonumber  
&&\lim_{n\to +\infty}\hh^{s_n}(0,E_n)- \frac{d\omega_d}{s_n}\\ \nonumber
&=& \lim_{n\to +\infty}  \lim_{r\to 0^+}  \int_{B_{\eta} \setminus B_r} \frac{\chi_{\R^d\setminus   E_n}(y)-\chi_{  E_n}(y)}{| y|^{d+s_n}}\ud y+d\omega_d \lim_{n\to +\infty} \frac{\eta^{-s_n}-1}{s_n}\\ &=& \lim_{n\to +\infty}  \lim_{r\to 0^+}  \int_{B_{\eta} \setminus B_r} \frac{\chi_{\R^d\setminus   E_n}(y)-\chi_{  E_n}(y)}{| y|^{d+s_n}}\ud y-d\omega_d\log\eta.\label{s}
\end{eqnarray}
In view of \eqref{uni} we have 
\begin{equation}\label{quasifine}
\begin{aligned}
&\lim_{n\to +\infty}\lim_{r\to 0^+} \int_{B_{\eta} \setminus B_r} \frac{\chi_{\R^d\setminus   E_n}(y)-\chi_{E_n}(y)}{| y|^{d+s_n}}\ud y\\
=&\lim_{n\to +\infty}\int_{B_{\eta} \setminus (B_\delta(-\delta e_d)\cup  B_\delta(\delta e_d)) } \frac{\chi_{\R^d\setminus   E_n}(y)-\chi_{  E_n}(y)}{| y|^{d+s_n}}\ud y\\
=&\int_{B_{\eta} \setminus (B_\delta(-\delta e_d)\cup  B_\delta(\delta e_d)) } \frac{\chi_{\R^d\setminus   E}(y)-\chi_{  E}(y)}{| y|^{d}}\ud y\,,
\end{aligned}
\end{equation}
where the last equality follows by Lemma \ref{inou} and by the Dominated Convergence Theorem.
 By \eqref{s}, \eqref{quasifine} and by \eqref{uni},
 we get
 \begin{equation*}
 \begin{aligned}
 \lim_{n\to+\infty} \hh^{s_n}(0, E_n)-\frac{d\omega_d}{s_n}=&\int_{B_{\eta} \setminus (B_\delta(-\delta e_d)\cup  B_\delta(\delta e_d)) } \frac{\chi_{\R^d\setminus   E}(y)-\chi_{  E}(y)}{| y|^{d}}\ud y-d\omega_d\log\eta\\
=& \lim_{r\to 0^+}  \int_{B_{\eta} \setminus B_r} \frac{\chi_{\R^d\setminus   E}(y)-\chi_{ E}(y)}{| y|^{d}}\ud y-d\omega_d\log\eta\,,
\end{aligned}
\end{equation*}
 which, in view of \eqref{noRtoin}, implies \eqref{order1}.
\end{proof} 
In the next two results we characterize the limit of $s$-fractional mean curvature flows as $s\to 0$.

\begin{theorem}\label{flowconvsto0ord0}
Let $\{s_n\}_{n\in\N}\subset (0,1)$ with $s_n\to 0$ as $n\to +\infty$. 
Let $u_0\in C(\R^d)$ be a uniformly continuous function, constant outside a compact set. 

For every $n\in\N$ let $u^n$ be the viscosity solution to \eqref{levelsetf} with $\H$ replaced by $\hh^{s_n}$, and set $v^n(x,t):= u^n(x,{s_n}t)$ for all $x\in\R^d,\, t\ge 0$. 
Then, $v^n\to v^\infty$ locally uniformly where $v^\infty:\R^d\times [0,+\infty)\to\R$ is the (unique) viscosity solution to \eqref{levelsetf} with $\H$ replaced by $d \omega_d$.
\end{theorem}
\begin{proof}
For every $n\in\N$ we set $\H^n:=s_n\hh^{s_n}$ and $\H^\infty:=d\omega_d$.
By Proposition \ref{properties}, $\H^n$ are nonlocal curvatures in the sense of Subsection \ref{axioms} and satisfy (C') and (S). Trivially, also $\H^\infty$ is a nonlocal curvature satisfying (C') and (S).
Moreover, $\H^n$ are positive on every ball of radius $\rho>0$; furthermore, by the scaling property
\begin{equation}\label{sscaling}
\hh^{s}(\lambda x,\lambda E)=\lambda^{-s}\hh^{s}(x,E)\qquad\textrm{for all }\lambda>0,s\in(0,1),\,E\in\Reg,\, x\in\partial E,
\end{equation}
we deduce that 
$$
\H^n(\rho x,\overline B_\rho)=\rho^{-s_n} \H^n(x,\overline B_1)\qquad\textrm{for all }\rho>0, n\in\N,\, x\in\partial B_1\,.
$$
Therefore, $0\leq \H^n(\rho x,\overline{B}_\rho)\leq \max (1, \rho^{-1}) \sup_n\H^n(x,\overline{B}_1)$ 
and so in view of Theorem  \ref{convergence0}, the sequence $\{\H^n\}_{n\in\N}$ satisfies also property (UB).  Again by Theorem \ref{convergence0} (in particular, by \eqref{order0}), we get that $\H^n\to\H^\infty$ in the sense of Definition \ref{curvconver}. 

One can easily check that $v^n$ are viscosity solutions to \eqref{levelsetf} with $\H$ replaced by $\H^n$, so that,
by Theorem \ref{genthm} we can conclude that $v^n\to v^\infty$ locally uniformly, where $v^\infty$ is the viscosity solution to \eqref{levelsetf} with $\H$ replaced by $\H^\infty$.
\end{proof}
\begin{theorem}\label{flowconvsto0ord1}
Let $\{s_n\}_{n\in\N}\subset (0,1)$ with $s_n\to 0$ as $n\to +\infty$. 
Let $u_0\in C(\R^d)$ be a uniformly continuous function, constant outside a compact set. 

For every $n\in\N$ let $u^n$ be the viscosity solution to \eqref{levelsetf} with $\H$ replaced by $\hh^{s_n}-\frac{d\omega_d}{s_n}$.
Then, $u^n\to u^\infty$ locally uniformly where $u^\infty:\R^d\times [0,+\infty)\to\R$ is the (unique) viscosity solution to \eqref{levelsetf} with $\H$ replaced by $\hh^0$.
\end{theorem}
\begin{proof}
For every $n\in\N$ we set $\H^n:=\hh^{s_n}-\frac{d\omega_d}{s_n}$ and $\H^\infty:=\hh^0$.
By Proposition \ref{properties}, $\H^n$ and $\H^\infty$ are nonlocal curvatures in the sense of Subsection \ref{axioms} and satisfy (C') and (S). 
In view of \eqref{sscaling}, we have
\begin{equation}\label{lu1}
\H^n(\rho x,\overline B_\rho)=\rho^{-s_n}\H^n(x,\overline B_1)+d\omega_d\frac{\rho^{-s_n}-1}{s_n}\quad\textrm{for all }\rho>0,\,x\in\partial B_1.
\end{equation}
For $\rho>1$, by Lagrange Theorem, for every $n\in\N$ there exists $\xi_n\in(0,s_n)$ such that
\begin{equation}\label{lu2}
\frac{\rho^{-s_n}-1}{s_n}=-\rho^{-\xi_n}\log\rho\geq  -\log\rho
\end{equation} 
therefore, by \eqref{lu1}, \eqref{lu2},  and Theorem \ref{convergence0},  we have that there exists a constant $K\ge 0$ such that 
\begin{equation}\label{1ub}
\uc^{\inf}(\rho):=\inf_{n\in\N}\uc^n(\rho)\ge \inf_n \rho^{-s_n}\H^n(x,\overline{B}_1)-d\omega_d\log \rho\ge -K\rho\qquad\textrm{ for all }\rho>1\,.
\end{equation}
Moreover, again by \eqref{lu1}, \eqref{lu2}, and Theorem \ref{convergence0}, it is easy to see that there exist two constants $C_1,C_2>0$ such that
\begin{equation}\label{2ub}
\oc^{\sup}(\rho):=\sup_{n\in\N}\oc^n(\rho)
\le C_1\frac {1+|\log\rho| }\rho+C_2\quad\textrm{for all }\rho>0\,.
\end{equation}
Therefore, by \eqref{1ub} and \eqref{2ub} we deduce that $\{\H^n\}_{n\in\N}$ satisfies also property (UB), and again by Theorem \ref{convergence0} (in particular, by \eqref{order1}), we get that $\H^n\to\H^\infty$ in the sense of Definition \ref{curvconver}. 
One can thus apply Theorem \ref{genthm} in order to get the claim.
\end{proof}
\subsection{Convergence of the $s$-fractional mean curvature flows as $s\to 1^-$}

For every $E\in\Reg$, and for every $x\in\partial E$, we denote by $\hh^1(x,E)$  the scalar mean curvature of the set $\partial E$ at $x$, i.e., the sum of the principal curvatures of $\partial E$ at $x$. Notice that, since $E\in\Reg$, near $x=(x',x_d)$ the boundary of $E$ can be described, in suitable coordinates,  as the graph of a function $f$ in $C^{\ell, \beta}$, with $f(x')=0, \, \D f(x')=0$; then, denoting by $A_x$ the $(d-1)\times(d-1)$ Hessian matrix of $f$ at a point $x=(x',f(x'))$, 
it is well known that
\begin{equation}\label{coord}
\hh^1(x,E)=-\frac{1}{d(d-1)\omega_{d}}\int_{\mathbb{S}^{d-2}}
 e^t A_x e \ud\mathcal{H}^{d-2}(e),
\end{equation}
where   $\mathbb{S}^{d-2}$ stands for the unit sphere in $\R^{d-1}$.

\begin{theorem}\label{convergence1}
Let $\{s_n\}_{n\in\N}\subset(0,1)$ be such that $s_n\to 1$ as $n\to+\infty$. Let $\{E_n\}_{n\in\N}\subset\Reg$ be such that $E_n\to E$ in $\Reg$ for some $E\in\Reg$.
For every $x\in\partial E\cap \partial E_n$
\begin{equation}\label{convuno}
\lim_{n\to+\infty} (1-s_n)\hh^{s_n}(x,E_n)=d(d-1)\omega_{d} \hh^1(x,  E).
\end{equation}

\end{theorem}

\begin{proof} The pointwise convergence (i.e., for $E_n\equiv E$)  has been proved in \cite{av, cv, i}. By a direct inspection of the proof it is easy to check that actually the convergence is uniform. 
For the sake of completeness, we briefly sketch  the computation. 
Arguing as in Theorem \ref{convergence0},   we may assume that $x=0$ and  $\nu_{E}(0)=\nu_{E_n}(0)=e_d$, where for all $F\in\Reg$ and $y\in\partial F$ the symbol $\nu_F(y)$ denotes the outer  normal to $\partial F$ at $y$. 
Since 
\[
\text{div}\left(\frac{y}{|y|^{d+s}}\right)=-s\frac{1}{|y|^{d+s}},
\]
we may rewrite
\begin{equation}\label{iniz}
\hh^{s_n}(0,E_n)= \frac{ 2}{s_n}  \int_{ \partial E_n } \frac{\nu(y)\cdot y }{|y|^{d+s}}\ud\mathcal{H}^{d-1}(y)\,.
\end{equation}

For every $\delta>0$ the symbol $B'_\delta$ denotes the ball in $\R^{d-1}$ of center $0$ and radius $\delta$.

Since $E_n\to E$  in $\Reg$ there exist  $\delta>0$ and functions $f_n,f\in C^{\ell,\beta}(B'_\delta;[0,\delta))$ such that $f_n\to f$ in $C^{\ell,\beta}(B'_\delta;\R)$ such that $f_n(0)=f(0)=0$,  $\D f_n(0)=\D f(0)=0$ and  
\begin{eqnarray*}
&\partial E_n \cap B_\delta=\{(y',f_n(y'))\,:\, y'\in B'_\delta\},\quad &\partial E \cap B_\delta=\{(y',f(y'))\,:\, y'\in B'_\delta\}\\
&E_n \cap B_\delta=\{(y',y_d)\,:\, y'\in B'_\delta,\, y_d\leq f_n(y')\},\quad &E \cap B_\delta=\{(y',y_d)\,:\, y'\in B'_\delta,\, y_d\leq f(y')\}.
\end{eqnarray*}  
Let  $\eta>\dia (\partial E)$, so that for $n$ large enough $\partial E_n,\partial E\subset B_\eta$. 

On the one hand, we notice that
\begin{equation}\label{fuori1s1} 
\left|\frac{ 2}{s_n}  \int_{ \partial E_n \setminus B_\delta } \frac{\nu(y)\cdot y }{|y|^{d+s_n}}\ud\mathcal{H}^{d-1}(y)\right|  \leq  \frac{2}{s\delta^{d+s_n-1}}\text{Per}(E_n)\le \frac{C}{s\delta^{d+s_n-1}}\,,  
\end{equation}
for some constant $C>0$ independent of $n$.
On the other hand we have 
\begin{equation}\label{contobordo} 
\frac{2}{s_n} \int_{\partial E_n \cap B_\delta } \frac{\nu_{E_n}(y)\cdot y}{|y|^{d+s_n}}\ud\mathcal{H}^{d-1}(y) =\frac{2}{s_n} \int_{B'_\delta} \frac{f_n(y')- \D f_n(y')\cdot y'}{(f_n^2(y')+|y'|^2)^{\frac{d+s_n}{2}}}\ud y'.
\end{equation}

Since $f_n(y')=\frac{1}{2} (y')^t\D^2 f_n(0) y' +o(|y'|^2)$ and $\D f_n(y')= \D^2 f_n(0)y' +\mathrm{o}(|y'|)$, letting $e=\frac{y'}{|y'|}$, we have 
that \begin{eqnarray}\label{contof}
f_n(y')-\D f_n(y')\cdot y'&=&-\frac{1}{2}|y'|^2 (e^t \D^2 f_n(0)e) +\mathrm{o}(|y'|^2) 
\\ \label{contotay}
 \frac{1}{(f_n^2(y')+|y'|^2)^{\frac{d+s_n}{2}}} &=& \frac{1}{|y'|^{d+s_n}\left[1+\frac{1}{4}|y'|^2 ( e^t\D^2 f_n(0) e)^2+\mathrm{o}(|y'|^2)\right]^{\frac{d+s_n}{2}}}
  \\
\nonumber &=& \frac{1}{|y'|^{d+s_n}} \left[1-\frac{d+s_n}{8}|y'|^2 (e^t\D^2 f_n(0) e)^2+\mathrm{o}(|y'|^2)\right].
\end{eqnarray}
 Replacing \eqref{contof} and \eqref{contotay} in \eqref{contobordo} we get 
 \begin{eqnarray}\label{contobordonuovo}
 && \frac{2}{s_n} \int_{\partial E_n \cap B_\delta } \frac{\nu_{E_n}(y)\cdot y}{|y|^{d+s_n}}d\mathcal{H}^{d-1}(y) \\
 \nonumber &=& -\frac{1}{s_n} \int_{B'_\delta}
 \frac{e^t\D^2 f_n(0) e}{|y'|^{d+s_n-2}} \left[1-\frac{d+s_n}{8}|y'|^2 (e^t\D^2 f_n(0) e)^2+\mathrm{o}(|y'|^2)\right]\ud y'\\\nonumber 
 &=& 
-\frac{1}{s_n(1-s_n)}\delta^{1-s_n}\int_{\mathbb{S}^{d-2}} 
 e^t\D^2 f_n(0) e \ud\mathcal{H}^{d-2}(e)
\\ \nonumber &+&
 \frac{d+s_n}{8s_n(3-s_n)}\delta^{3-s_n}\int_{\mathbb{S}^{d-2}} 
( e^t \D^2 f_n(0) e)^3 \ud \mathcal{H}^{d-2}(e)+\mathrm{o}(1).
\end{eqnarray}
Therefore, \eqref{iniz}, \eqref{fuori1s1} and  \eqref{contobordonuovo}, together with the $C^{\regu}$-convergence of $f_n$ to $f$ and \eqref{coord}, imply \eqref{convuno}.

%
%
%
%
\end{proof} 

\begin{theorem}\label{flowconvsto1}
Let $\{s_n\}_{n\in\N}\subset (0,1)$ with $s_n\to 1$ as $n\to +\infty$. 
Let $u_0\in C(\R^d)$ be a uniformly continuous function, constant outside a compact set. 

For every $n\in\N$ let $u^n$ be the viscosity solution to \eqref{levelsetf} with $\H$ replaced by $\hh^{s_n}$, and set $v^n(x,t):= u^n(x,(1-s_n)t)$ for all $x\in\R^d,\, t\ge 0$. 
Then, $v^n\to v^\infty$ locally uniformly where $v^\infty:\R^d\times [0,+\infty)\to\R$ is the (unique) viscosity solution to \eqref{levelsetf} with $\H$ replaced by $d(d-1)\omega_d\hh^1$.
\end{theorem}
\begin{proof}
For every $n\in\N$ we set $\H^n:=(1-s_n)\hh^{s_n}$ and $\H^\infty:=d(d-1)\omega_d\hh^1$.
By Proposition \ref{properties}, $\H^n$ are nonlocal curvatures in the sense of Subsection \ref{axioms} and satisfy (C') and (S). Trivially, also $\H^\infty$ is a nonlocal curvature satisfying (C') and (S).
Moreover, $\H^n$ are positive on all the balls of radius $\rho>0$; furthermore, by the scaling property \eqref{sscaling},
we deduce that 
$$
\H^n(\rho x,\overline B_\rho)=\rho^{-s_n} \H^n(x,\overline B_1)\qquad\textrm{for all }\rho>0, n\in\N,\, x\in\partial B_1\,.
$$
Therefore,  in view of Theorem  \ref{convergence1}, the sequence $\{\H^n\}_{n\in\N}$ satisfies also property (UB) and again by Theorem \ref{convergence1}, we get that $\H^n\to\H^\infty$ in the sense of Definition \ref{curvconver}. 

One can easily check that $v^n$ are viscosity solutions to \eqref{levelsetf} with $\H$ replaced by $\H^n$, so that,
by Theorem \ref{genthm}, we can conclude that $v^n\to v^\infty$ locally uniformly, where $v^\infty$ is the viscosity solution to \eqref{levelsetf} with $\H$ replaced by $\H^\infty$.

\end{proof}

\begin{remark}\upshape\label{remani}  
There exist also  anisotropic versions of the $s$-fractional curvature. Let $K\subseteq \R^d $ be a 
 compact convex set with non-empty interior and symmetric with respect to the origin,  and  let $|x|_K:=\inf\{\lambda:  x\in \lambda K\}$, i.e.,  the norm in $\R^d$ having the set $K$ as unitary ball.  
 
For every $s\in (0,1)$, the anisotropic $s$-fractional perimeter of a measurable set $E\subset\R^d$ is defined by
\begin{equation}\label{anis}
 \P^s_{K}(E):=\int_{E}\int_{\R^d\setminus E} \frac{1}{|x-y|^{d+s}_{K}} \ud y\ud x
 \end{equation}
 and its first variation, i.e., the anisotropic $s$-fractional curvature \cite{cnr}
is formally given by 
\begin{equation}\label{an}
\hh^s_K(x, E):=   \lim_{r\to 0^+}  \int_{\Rd\setminus B_r(x)} \frac{\chi_{\R^d\setminus E}(y)-\chi_{E}(y)}{|x-y|_K^{d+s}}\ud y, \qquad x\in \partial E.
\end{equation}
It is easy to check that Proposition \ref{properties}  applies  also to these curvatures (see \cite{cnr}), hence they are nonlocal curvatures in the sense of Subsection \ref{axioms}.  

In \cite{l}, using suitable integration formulas,  the convergence (pointwise and in the sense of $\Gamma$-convergence) of $\P^s_K$ to $\P_{ZK}$
 as $s\to 1$ is proved.  The limit $\P_{ZK}$ is given by
\begin{equation} \label{peran} 
\P_{ZK}(E)=\int_{\partial^* E} |\nu_E(x)|_{Z^*K} \ud\mathcal{H}^{d-1}(x)
\end{equation}
where $|y|_{Z^*K}:=\frac{d+1}{2} \int_K |y\cdot z|\ud z.$  It is easy to check that  $Z^\star K=\{ x: |x|_{Z^*K}\leq 1\}$   is strictly convex. 

We expect that by similar methods it is possible to  show that $(1-s)\hh^{s}_K$ converge in the sense of Definition \ref{curvconver}  as  $s\to 1^-$  to a  local anisotropic curvature $\hh_{ZK}^1$,  which is given  by the first variation of the anisotropic perimeter  \eqref{peran}.  Since  $Z^\star K$ is strictly convex, then $\hh_{ZK}^1$ is a  curvature in the sense of  Subsection \ref{axioms}, and  (UB) is trivially satisfied, so Theorem \ref{genthm} on the convergence  of the corresponding geometric flows should also apply. 
 \end{remark} 


\section{The Riesz curvature flow}\label{secriesz}
In this section we introduce and analyze a new nonlocal geometric flow, where the curvature  is the first variation of a Riesz interaction energy. 

Let $s\in (-d,0)$\,.
For every $E\in\Reg$ and for every $x\in\partial E$, we set
\begin{equation}\label{rieszvar} 
\kk^s(x,E):= \begin{cases} \displaystyle - 2\int_E  \frac{1}{|x-y|^{d+s}}\ud y\,,& E\subset\subset\R^d\,,\\ 
\quad&\\
\displaystyle 2\int_{\R^d\setminus E}  \frac{1}{|x-y|^{d+s}}\ud y\,,& \R^d\setminus E\subset\subset\R^d\,.
\end{cases}
\end{equation}  
Note that, for $E$ compact, $\kk^s$ is the first variation of the perimeter-like Riesz interaction functional 
\begin{equation}\label{riesz} 
\J^s(E):= - \int_E\int_E \frac{1}{|x-y|^{d+s}}\ud y\ud x\,. 
\end{equation} 
First of all we observe the following. 
\begin{proposition}\label{propertiesriesz}
For every $s\in (-d,0)$ the functionals $\kk^s$ satisfy the properties (M), (T), (C'), and (S) in Subsection \ref{axioms}. Moreover, if $s\in  [-1,0)$, then $\kk^s$ satisfies also property (B). 
\end{proposition}
\begin{proof}The validity of properties  (M), (T), (C'), and (S)  can be proven by arguing exactly as in Proposition \ref{properties}. As for (B), we observe, denoting with $B'_\rho$ the ball in $\R^{d-1}$ with radius $\rho$, that 
\begin{align*} \kk^s(0,B_\rho(\rho e_d))= & - 2\int_{B_\rho(\rho e_d)}  \frac{1}{|y|^{d+s}}\ud y\\ =& -2 \int_{B'_\rho}  \int_{\rho-\sqrt{\rho^2-|y'|^2}}^\rho\frac{1}{(|y'|^2+y_d^2)^{\frac{d+s}{2}}}dy_d\ud y' - 2\int_{B_\rho(\rho e_d)\cap \{y_d\geq \rho\}}  \frac{1}{|y|^{d+s}}\ud y\\
\geq & -2 \int_{B'_\rho} \frac{\sqrt{\rho^2-|y'|^2}}{|y'|^{d+s}}\ud y' -\frac{\omega_d}{\rho^s} \geq \frac{2\omega_{d-1}\rho }{s\rho^{s}} -\frac{\omega_d}{\rho^s}.  \end{align*} Therefore (B) is satisfied for $s\in [-1,0)$. 
\end{proof}
Existence of global solution and uniqueness of  $s$-Riesz curvature flows for $s\in [-1,0)$ is established by the following result that is a direct consequence of Theorem \ref{exun} and of Proposition \ref{propertiesriesz}.

\begin{theorem}\label{exunsneg}
For every $s\in [-1,0)$ and for every uniformly continuous function $u_0\in C(\R^d)$ constant outside a compact set, there exists a unique viscosity solution to \eqref{levelsetf} with $\H$ replaced by $\kk^s$.
\end{theorem}

\begin{remark}\upshape Note that there holds 
\[\kk^s(0,B_\rho(\rho e_d))\leq  - 2\int_{B_\rho(\rho e_d)\cap \{y_d\geq \rho\}}  \frac{1}{|y|^{d+s}}\ud y \leq -\frac{\omega_d}{5^{(d+s)/2}\rho^s}. \]
Therefore,  for $s\in (-d,-1)$, balls blow up in finite time. Nevertheless, in this case one could prove at least local in time existence of the viscosity solution.
\end{remark}
 
\begin{theorem}\label{convergenceriesz}
Let $\{s_n\}_{n\in\N}\subset(-d,0)$ be such that $s_n\to 0^-$ as $n\to+\infty$. Let $\{E_n\}_{n\in\N}\subset\Reg$ be such that $E_n\to E$ in $\Reg$ for some $E\in\Reg$.
For every $x\in\partial E\cap \partial E_n$ it holds 
\begin{equation}\label{order1riesz}
\lim_{n\to +\infty} \kk^{s_n}(x, E_n)-\frac{d\omega_d}{s_n} = \hh^0(x, E),
\end{equation}
 where $\hh^0$ is defined in \eqref{zerocurv}. 
In particular,
\begin{equation}\label{order0riesz}
\lim_{n\to +\infty} s_n\kk^{s_n}(x,E_n)=d\omega_d.
\end{equation}
\end{theorem}
\begin{proof} 
The proof follows along the lines of that of Theorem \ref{convergence0}; we briefly sketch it. 

Let  $x\in \partial E_n\cap \partial E$ for all $n\in\N$.
First, notice that all the curvatures we are dealing with satisfy assumption (S) and are invariant by rotations (and translations). In particular,  we can assume  without loss of generality that  $E_n$ and $E$ are compact, that $x=0$ 
 and  $\nu_E(0)= \nu_{E_n} (0) =e_d$, where we recall that for all $F\in\Reg$ and $y\in\partial F$, $\nu_F(y)$ denotes the outer  normal to $\partial F$ at $y$.

Setting $\eta:= 2 \dia (E)$,
we have $E, \, E_n \subset B_\eta$ for $n$ large enough. 
Since $ E_n \to E$ in $\Reg$,   we deduce that $E_n$ and $E$   satisfy a uniform interior and exterior ball condition. More precisely,  there exists  $\delta>0$ such that $B_\delta(-\delta e_d )\subseteq E_n$ and $B_\delta( \delta e_d)\subseteq\Rd\setminus E_n$ and the same for $E$. Due to this fact, and to the symmetry property of the kernel, we get
 \begin{equation}\label{riesz1}
 \begin{aligned}  
 &\kk^{s_n}(0,E_n)
 = - 2\int_{E_n} \frac{1}{| y|^{d+s_n}}\ud y\\
 =& -\int_{B_\delta( -\delta e_d)\cup B_\delta( \delta e_d)}  \frac{1}{| y|^{d+s_n}}\ud y-2\int_{E_n\setminus B_\delta (-\delta e_d)}  \frac{1}{| y|^{d+s_n}}\ud y\\
=&- \int_{B_\eta} \frac{1}{|y|^{d+s_n}}\ud y 
+\int_{B_\eta\setminus (B_\delta( -\delta e_d)\cup B_\delta( \delta e_d)) } \frac{1}{|y|^{d+s_n}}\ud y-2\int_{E_n\setminus B_\delta (-\delta e_d)}  \frac{1}{| y|^{d+s_n}}\ud y\\
=&- \int_{B_\eta} \frac{1}{|y|^{d+s_n}}\ud y 
+\int_{B_\eta\setminus (B_\delta( -\delta e_d)\cup B_\delta( \delta e_d)) } \frac{\chi_{\R^d\setminus E_n}(y)-\chi_{E_n}(y)}{|y|^{d+s_n}}\ud y\\
 =&\frac{d\omega_{d}\eta^{-s_n}}{s_n}+\int_{B_\eta\setminus (B_\delta( -\delta e_d)\cup B_\delta( \delta e_d))} \frac{\chi_{\R^d\setminus E_n}(y)-\chi_{E_n}(y)}{|y|^{d+s_n}}\ud y\,.
\end{aligned}
 \end{equation}  
By Lemma \ref{inou} and using the $C^{\regu}$ convergence of $E_n$ to $E$ and the Dominate Convergence Theorem, we get
\begin{equation}\label{riesz2}
\begin{aligned}
&\lim_{n\to +\infty} \int_{B_\eta\setminus (B_\delta( -\delta e_d)\cup B_\delta( \delta e_d))} \frac{\chi_{\R^d\setminus E_n}(y)-\chi_{E_n}(y)}{|y|^{d+s_n}}\ud y\\
=&
\int_{B_\eta\setminus B_\delta( -\delta e_d)\cup B_\delta( \delta e_d) } \frac{\chi_{\R^d\setminus E}(y)-\chi_E(y)}{|y|^{d}}\ud
y\\
=&\lim_{r\to 0^+} \int_{B_\eta\setminus B_r } \frac{\chi_{\R^d\setminus E}(y)-\chi_E(y)}{|y|^{d}}\ud
y\,,
\end{aligned}
\end{equation}
 where the last equality follows from Lemma \ref{inou2}.
 
 By \eqref{riesz1} and \eqref{riesz2}, in view of \eqref{noRtoin}, we can conclude that \eqref{order1riesz} holds true.
\end{proof}  
Next two results are devoted to the analysis of the $s$-Riesz curvature flow as $s\to 0$.
\begin{theorem}\label{flowconvrieszsto0ord0}
Let $\{s_n\}_{n\in\N}\subset [-1,0)$ with $s_n\to 0$ as $n\to +\infty$. 
Let $u_0\in C(\R^d)$ be a uniformly continuous function, constant outside a compact set. 

For every $n\in\N$ let $u^n$ be the viscosity solution to \eqref{levelsetf} with $\H$ replaced by $\kk^{s_n}$, and set $v^n(x,t):= u^n(x,-s_n t)$ for all $x\in\R^d,\, t\ge 0$. 
Then, $v^n\to v^\infty$ locally uniformly where $v^\infty:\R^d\times [0,+\infty)\to\R$ is the (unique) viscosity solution to \eqref{levelsetf} with $\H$ replaced by $-d \omega_d$.
\end{theorem}
\begin{proof}
For every $n\in\N$ we set $\H^n:=-s_n\kk^{s_n}$ and $\H^\infty:=-d\omega_d$.
By Proposition \ref{propertiesriesz}, $\H^n$ are nonlocal curvatures in the sense of Subsection \ref{axioms} and satisfy (C') and (S). Trivially, also $\H^\infty$ is a nonlocal curvature satisfying (C') and (S).
Moreover, from the scaling property
\begin{equation}\label{snegscaling}
\kk^{s}(\lambda x,\lambda E)=\lambda^{-s}\kk^{s}(x,E)\qquad\textrm{for all }s\in(-d,0),\,E\in\Reg,\, x\in\partial E,
\end{equation}
we deduce that 
$$
\H^n(\rho x,\overline B_\rho)=\rho^{-s_n} \H^n(x,\overline B_1)\qquad\textrm{for all }\rho>0, n\in\N,\, x\in\partial B_1\,.
$$
Since $\H^n$ is negative and $-s_n\in (0,1]$, it follows that
$$
\max\{\rho,1\}\H^n(x,\overline{B}_1)\le \H^n(\rho x,\overline{B}_\rho)\le \min\{\rho,1\}\H^n(x,\overline{B}_1)\quad\textrm{for all }\rho>0, n\in\N,\, x\in\partial B_1,
$$
which, in view of Theorem  \ref{convergenceriesz}, implies that the sequence $\{\H^n\}_{n\in\N}$ satisfies also property (UB).
Again by Theorem \ref{convergenceriesz} (in particular, by \eqref{order0riesz}), we get that $\H^n\to\H^\infty$ in the sense of Definition \ref{curvconver}. 

One can easily check that $v^n$ are viscosity solutions to \eqref{levelsetf} with $\H$ replaced by $\H^n$, so that,
by Theorem \ref{genthm} we can conclude that $v^n\to v^\infty$ locally uniformly, where $v^\infty$ is the viscosity solution to \eqref{levelsetf} with $\H$ replaced by $\H^\infty$.
\end{proof}
\begin{theorem}\label{flowconvriezsto0ord1}
Let $\{s_n\}_{n\in\N}\subset [-1,0)$ with $s_n\to 0$ as $n\to +\infty$. 
Let $u_0\in C(\R^d)$ be a uniformly continuous function, constant outside a compact set. 

For every $n\in\N$ let $u^n$ be the viscosity solution to \eqref{levelsetf} with $\H$ replaced by $\kk^{s_n}-\frac{d\omega_d}{s_n}$.
Then, $u^n\to u^\infty$ locally uniformly where $u^\infty:\R^d\times [0,+\infty)\to\R$ is the (unique) viscosity solution to \eqref{levelsetf} with $\H$ replaced by $\hh^0$.
\end{theorem}
\begin{proof}
For every $n\in\N$ we set $\H^n:=\kk^{s_n}-\frac{d\omega_d}{s_n}$ and $\H^\infty:=\hh^0$.
By Proposition \ref{propertiesriesz}, $\H^n$ and $\H^\infty$ are nonlocal curvatures in the sense of Subsection \ref{axioms} and satisfy (C') and (S). 
In view of \eqref{snegscaling}, we have
\begin{equation}\label{lu1riesz}
\H^n(\rho x,\overline B_\rho)=\rho^{-s_n}\H^n(x,\overline B_1)+d\omega_d\frac{\rho^{-s_n}-1}{s_n}\quad\textrm{for all }\rho>0,\,x\in\partial B_1.
\end{equation}
For $n$ large enough we have $s_n\ge -\frac12$, and hence, 
for $\rho>1$, by Lagrange Theorem, there exists $\xi_n\in(-\frac 12,0)$ such that
\begin{equation}\label{lu2riesz}
\frac{\rho^{-s_n}-1}{s_n}=-\rho^{-\xi_n}\log\rho \ge -\rho^{\frac 1 2}\log\rho;
\end{equation} 
therefore, by \eqref{lu1riesz}, \eqref{lu2riesz}, Theorem \ref{convergenceriesz}, and \eqref{ball}, we have that there exists a constant $K>0$ such that
\begin{equation}\label{1ubriesz}
\uc^{\inf}(\rho):=\inf_{n\in\N}\uc^n(\rho)\ge -K\rho\qquad\textrm{ for all }\rho>1\,.
\end{equation}
Moreover, again by \eqref{lu1riesz}, \eqref{lu2riesz}, and Theorem \ref{convergenceriesz}, it is easy to see that there exist two constants $C_1,C_2>0$ such that
\begin{equation}\label{2ubriesz}
\oc^{\sup}(\rho):=\sup_{n\in\N}\oc^n(\rho)
\le C_1(\rho+|\log\rho|) +C_2\quad\textrm{for all }\rho>0\,.
\end{equation}
Therefore, by \eqref{1ubriesz} and \eqref{2ubriesz} we deduce that $\{\H^n\}_{n\in\N}$ satisfies also property (UB), and again by Theorem \ref{convergenceriesz} (in particular, by \eqref{order1riesz}), we get that $\H^n\to\H^\infty$ in the sense of Definition \ref{curvconver}. 
One can thus apply Theorem \ref{genthm} in order to get the claim.
\end{proof}

 \section{The flow generated by the regularized $r$-Minkowski content} \label{secmin} 
 As a final example,  we consider the asymptotic behavior of the  flow generated by  the regularized  $r$-Minkowski content   introduced in \cite{b}  in the framework of two-phase image segmentation. This flow has been considered also in \cite{CMP0,cmp} (see also \cite{dnv}) where, in particular,  it has been proved existence and uniqueness of the corresponding level set solution. 
 
 Let $r>0$ be fixed. For every measurable set $E$, we define the \emph{$r$-Minkowsky content of $E$} as
\begin{equation}\label{defmink}
J_r(E)\ :=\ \frac{1}{2r}\int_{\R^d} \textup{osc}_{B_r(x)}(\chi_E)\ud x\,,
\end{equation}
where  
$\textup{osc}_A(u)\ =\ \textup{ess} \sup_A u\,-\, \textup{ess} \inf_A u$.
One can check (see for instance \cite{CMP0}) that $J_r(E)$ coincides with the measure of the $r$-neighborhood
of the  essential boundary of $E$ divided by $2r$; moreover, one can show that, under mild regularity assumptions on $E$, $J_r(E)$ converges (pointwise and in sense of $\Gamma$-convergence) to the standard perimeter.  In \cite{CMP0} it has been proved that \eqref{defmink} is a generalized perimeter, and the corresponding curvature, i.e., its first variation with respect to inner variations, has been introduced. 
 
Let $E\in\Reg$.
For every $x\in\partial E$, we set
\begin{equation}\label{kappaprho}
\kappa_r(x,E)\ =\ \kappa^{\ou}_r(x,E)\,+\,\kappa^{\inn}_r(x,E),
\end{equation}
where
\begin{equation*} 
\kappa^{\ou}_r(x,E)\ =\begin{cases}
  \frac{1}{2r}\det (I+r \D \nu_E(x))\  &
\textrm{ if }  \ud(x+r \nu_E(x),E)= r\,,\\   
0 & \textrm{ otherwise,}
\end{cases}
\end{equation*}
\begin{equation*}
\kappa^{\inn}_r(x,E)\ =\begin{cases}
 - \frac{1}{2 r }\det (I-r \D \nu_E(x)) &
\textrm{ if } \ud(x -r \nu_E(x),\R^d\setminus E)=r\,. \\  0 & \textrm{ otherwise.}
\end{cases}
\end{equation*}
In the above formulas, $\nu_F(x)$ denotes the outer normal unit vector to $\partial F$ at $x$ whereas $\ud(y,F)$ is the distance between the point $y$ and the set $F$.

The curvature $\kappa_r(x,E)$ is not continuous and it is the true first variation of $J_r$ only for a strict subset of  $\Reg$. 
 In order to deal with a well defined curvature for all $E\in\Reg$, in \cite{CMP0}, a regularization of $J_r$, by an averaging procedure, has been defined as follows.  Fix a function  $f:\R\to [0,+\infty)$  which is  even, smooth and nonincreasing in  $[0,+\infty)$, with support in $[-1,1]$ and define $f_r(s):= \frac{1}{2r}f(s/r)$.   Set
\begin{equation}\label{varmr}
J^f_r(E)\ :=  \int_{\R^d} f_r(\ud_E(x))\ud x = \int_0^{r} (-2s f_r'(s)) J_s(E) \ud s=  \int_0^1(-s f'(s)) J_{r s}(E) \ud s ,
\end{equation}
where   $\ud_E$ is the signed distance from $\partial E$ and the second equality is obtained by exploiting coarea formula. 
 
For every $E\in\Reg, \, x\in\partial E$, the first variation of $J^f_r(E)$ at $x$ is now well defined \cite{CMP0} and given by  
\begin{equation}\label{kappaf}
\kappa_r^f(x,E):= \int_0^1 (- s f'(s)) \left[ \kappa^{\ou}_{rs}(x,E) +\kappa^{\inn}_{rs}(x,E)\right] \ud s.
\end{equation}
 
 Existence and uniqueness of flows driven by the curvature $\kappa_r^f$  have been  established in \cite[Section 6.4]{cmp} (see also \cite{CMP0}).  We briefly recall such results in next theorem. 
 
 \begin{theorem}\label{propertiesminkowski}
Let   $f:\R\to [0,+\infty)$  be   even, smooth and nonincreasing  in  $[0,+\infty)$, with support in $[-1,1]$. 
For every  $r>0$, the functionals $\kappa^f_r $ satisfy the properties (M), (T),  (C'), (B) and (S) in Subsection \ref{axioms}.  

For every $r>0$ and for every uniformly continuous function $u_0\in C(\R^d)$ constant outside a compact set, there exists a unique viscosity solution to \eqref{levelsetf} with $\H$ replaced by $\kappa^f_r $.
\end{theorem}

 
 We now  show the convergence of $\kappa^f_{r}$ as $r\to 0$ in the sense of Definition \ref{curvconver}. 
\begin{theorem}\label{convergencemin}  Let   $f:\R\to [0,+\infty)$  be   even, smooth and nonincreasing  in  $[0,+\infty)$, with support in $[-1,1]$. 
Let $\{r_n\}_{n\in\N}\subset(0,1)$ be such that $r_n\to 0$ as $n\to+\infty$. Let $\{E_n\}_{n\in\N}\subset\Reg$ be such that $E_n\to E$ in $\Reg$ for some $E\in\Reg$.
Then, for every $x\in\partial E\cap \partial E_n$ it holds 
\begin{equation}\label{min1}
\lim_{n\to +\infty} \kappa^f_{r_n}(x, E_n)  =  c_f \hh^1(x,E)\,,
\end{equation} 
where   $\hh^1(x,E)$ is   the scalar mean curvature of the set $\partial E$ at $x$ in \eqref{coord},  and 
$c_f=\int_0^1 f(s)\ud s$. 
\end{theorem}
\begin{proof} Let $r^{\inn}$ be the maximal radius $r$ such that $ E$  satisfies the interior ball condition with radius $r$, and let  $r^{\ou}$ be defined analogously. Clearly, $r^{\inn}$ and $r^{\ou}$ are  continuous with respect to smooth inner variations. For $n$ large enough the sets $E_n$ satisfy the interior and exterior ball condition with radius 
 $\bar r:= \min \{r^{\inn}, r^{\ou}\}/2$.

Fix  $s\in (0,1)$; for $n$ large enough  $r_n\le \bar r$, and hence   
 \[\kappa_{s r_n}(x,E_n) = \frac{\det (I+s r_n \D \nu_{E_n}(x))-\det (I-s r_n \D  \nu_{E_n}(x))}{2sr_n}= \textrm{tr}\D \nu_{E_n}(x)+\mathrm{o}(1).\]   
Since    $E_n\to E$  in $\Reg$,  we conclude that \[\lim_{r_n\to 0} \kappa_{s r_n}(x,E_n) =  \textrm{tr}D  \nu_{E}(x)  =  \hh^1(x,E)\,,
\]   
which, by the Dominated Convergence Theorem, implies 
\[\lim_{n\to +\infty} \kappa^f_{r_n}(x, E_n)  =  \hh^1(x,E) \int_0^1 (-s f'(s)) \ud s=\hh^1(x,E) \int_0^1 f(s) \ud s.  \]
 \end{proof} 

Finally, we conclude with the asymptotic result.
\begin{theorem}\label{flowminchio}  
Let   $f:\R\to[0,+\infty) $  be   even, smooth and nonincreasing  in  $[0,+\infty)$, with support in $[-1,1]$ and let $r_n\to 0$ as $n\to +\infty$. 
Let $u_0\in C(\R^d)$ be a uniformly continuous function, constant outside a compact set. 

For every $n\in\N$ let $u^n$ be the viscosity solution to \eqref{levelsetf} with $\H$ replaced by $\kappa^f_{r_n}$.
Then, $u^n\to u^\infty$ locally uniformly where $u^\infty:\R^d\times [0,+\infty)\to\R$ is the (unique) viscosity solution to \eqref{levelsetf} with $\H$ replaced by $c_f \hh^1$.
\end{theorem}
\begin{proof}
For every $n\in\N$ we set $\H^n:=\kappa^f_{r_n}$ and $\H^\infty:=\hh^1$.
By Theorem \ref{propertiesminkowski}, $\H^n$ are nonlocal curvatures in the sense of Subsection \ref{axioms} and satisfy (C') and (S). Trivially, also $\H^\infty$ is a nonlocal curvature satisfying (C') and (S). 

To check that   
 $\{\H^n\}_{n\in\N}$ satisfies  property (UB),  we note that 
 \begin{equation}\label{scalmi}
 \kappa^f_{r_n}(\rho x, B_\rho)= \frac{1}{\rho} \kappa^f_{\frac{r_n}{\rho}}( x, B_1)\quad\textrm{for all }n\in\N, \rho>0, x\in \partial B_1.
 \end{equation} Then  we compute for $r>0$, 
 \begin{equation}\label{nuovecu}
 \kappa_{r}^{\ou}(x,B_1)=  \frac{1}{2r} \left(1+r\right)^d    \ \forall r>0\quad \textrm{
and }\quad  
\kappa_{r}^{\inn}(x,B_1)=\begin{cases} - \frac{1}{2r} \left(1-r\right)^d   &r<1 \\ 0 &  r\geq 1 \end{cases}.  
\end{equation}
It is immediate to deduce that  $\kappa_{r}(x,B_1)\geq 0$ for all $r>0$ and then also $\kappa_{r}^f(x,B_1)\geq 0$ and, in view of \eqref{scalmi}, $\H^n(\rho x,\overline{B}_\rho)\ge 0$ for all $x\in\partial B_1$, $\rho>0$.
Moreover, by \eqref{nuovecu}, it follows that
\begin{align}\label{mink2} 
\kappa_{r}^{f}(x,B_1)\leq & \int_0^{1} (-sf'(s)) \frac{\left(1+sr \right)^d}{2sr} \ud s 
=    \frac{1}{2}\sum_{k=0}^d \left(\begin{array}{c} d\\ k \end{array}\right)  r^{k-1}\int_0^1 (- f'(s) )s^k \ud s\\ \nonumber 
\leq& C^f_1+C^f_2 r^{d-1},
\end{align}
for some constants $ C^f_1,C^f_2>0$ depending on $f$.
This fact, together with \eqref{scalmi}, implies that $\H^n$ satisfies also the second property of (UB) is satisfied.


%
Moreover,  by Theorem \ref{convergencemin} (in particular, by \eqref{min1}), we get that $\H^n\to\H^\infty$ in the sense of Definition \ref{curvconver}. 
One can thus apply Theorem \ref{genthm} in order to get the claim.
\end{proof}


\end{document}